\documentclass[11pt,a4paper,reqno]{amsart}
\usepackage{appendix,calc,amsfonts,amsthm,amscd,epsfig,psfrag,amsmath,amssymb,enumerate,graphicx}
\usepackage[showlabels,sections,floats,textmath,displaymath]{}
\reversemarginpar
\newlength\fullwidth
\numberwithin{equation}{section}

\DeclareMathSymbol{\leqslant}{\mathalpha}{AMSa}{"36} 
\DeclareMathSymbol{\geqslant}{\mathalpha}{AMSa}{"3E} 
\DeclareMathSymbol{\eset}{\mathalpha}{AMSb}{"3F}     
\renewcommand{\geq}{\;\geqslant\;}                   
\newcommand{\sumtwo}[2]{\sum_{\substack{#1 \\ #2}}} 

\def\1{\ifmmode {1\hskip -3pt \rm{I}} \else {\hbox {$1\hskip -3pt \rm{I}$}}\fi}

 \newcommand{\si}{\sigma } 
\newcommand{\be}{\begin{equation} } \newcommand{\tmix}{T_{\rm mix}} 

\newtheorem{Theorem}{Theorem}[section] 
\newtheorem{Lemma}[Theorem]{Lemma} 
\newtheorem{Proposition}[Theorem]{Proposition} 
\newtheorem{Corollary}[Theorem]{Corollary} 
\newtheorem{remark}[Theorem]{Remark}

\newtheorem{claim}[Theorem]{Claim}
\newtheorem{definition}[Theorem]{Definition}
 
\newcommand{\N}{\mathbb N}
\newcommand{\R}{\mathbb R}
\newcommand{\Z}{\mathbb Z}
\newcommand{\bP}{{\bf P}} 
\newcommand{\bE}{{\bf E}} 
\newcommand{\cA}{\ensuremath{\mathcal A}} 
\newcommand{\cB}{\ensuremath{\mathcal B}} 
 
\newcommand{\cD}{\ensuremath{\mathcal D}} 
\newcommand{\cE}{\ensuremath{\mathcal E}} 
\newcommand{\cF}{\ensuremath{\mathcal F}} 
\newcommand{\cG}{\ensuremath{\mathcal G}}

\newcommand{\cL}{\ensuremath{\mathcal L}} 
 
\newcommand{\cN}{\ensuremath{\mathcal N}} 
 
\newcommand{\cP}{\ensuremath{\mathcal P}}

\newcommand{\cS}{\ensuremath{\mathcal S}}

\newcommand{\cZ}{\ensuremath{\mathcal Z}} 
 
\newcommand{\bbN}{{\ensuremath{\mathbb N}} } 
 
\newcommand{\bbP}{{\ensuremath{\mathbb P}} } 
 
\newcommand{\bbR}{{\ensuremath{\mathbb R}} } 
\newcommand{\bbS}{{\ensuremath{\mathbb S}} }

\newcommand{\bbZ}{{\ensuremath{\mathbb Z}} } 

\let\a=\alpha \let\b=\beta   \let\d=\delta  \let\gep=\varepsilon
 \let\g=\gamma \let\h=\eta      
            
  \let\s=\sigma \let\t=\tau   
  
\let\D=\Delta   \let\G=\Gamma  \let\L=\Lambda 
\let\O=\Omega      

\def\\{\hfill\break}

\def\thsp{\thinspace}

\def\tthsp{\kern .083333 em}

\def\?{\mskip -10mu}
\newfam\msafam
\newfam\msbfam
\newfam\eufmfam
\def\hexnumber#1{%
  \ifcase#1 0\or 1\or 2\or 3\or 4\or 5\or 6\or 7\or 8\or
  9\or A\or B\or C\or D\or E\or F\fi}
\def\({\left(}
\def\){\right)}

\let\neper=e
\let\ii=i

\def\ie{\hbox{\it i.e.\ }}

\let\id=\identity

\let\sset=\subset

\def\nep#1{ \neper^{#1}}

\def\tc{\thsp | \thsp}
\let\<=\langle
\let\>=\rangle

\def\Var{ \mathop{\rm Var}\nolimits }

\def\gap{\mathop{\rm gap}\nolimits}

\def\inte#1{\lfloor #1 \rfloor}
\def\ceil#1{\lceil #1 \rceil}

\begin{document}
\title[]{On the mixing time of the 2D stochastic Ising model
with ``plus'' boundary conditions at low temperature}

\author {Fabio Martinelli} \address{ Dipartimento di Matematica,
  Universit\`a Roma Tre, Largo S.\ Murialdo 1, 00146 Roma, Italia.
  e--mail: {\tt martin@mat.uniroma3.it}} \author {Fabio Lucio
  Toninelli} \address{CNRS and ENS Lyon, Laboratoire de Physique\\
  46 All\'ee d'Italie, 69364 Lyon, France.  e--mail: {\tt
    fabio-lucio.toninelli@ens-lyon.fr}} \thanks{This work was
  supported by the European Research Council through the ``Advanced
  Grant'' PTRELSS 228032, and by ANR through the grants POLINTBIO and LHMSHE}

\begin{abstract}
  We consider the Glauber dynamics for the 2D Ising model in a box of
  side $L$, at inverse temperature $\b$ and random boundary conditions
  $\t$ whose distribution $\bP$ either stochastically dominates the
  extremal plus phase (hence the quotation marks in the title) or is
  stochastically dominated by the extremal minus phase. A particular
  case is when $\bP$ is concentrated on the homogeneous configuration
  identically equal to $+$ (equal to $-$). For $\b $ large enough we
  show that for any $\gep >0$ there exists $c=c(\b,\gep)$ such that
  the corresponding mixing time $T_{\rm mix}$ satisfies
  $\lim_{L\to\infty}\bP\(T_{\rm mix}\ge \exp({cL^\gep})\) =0
$. In the non-random case $\t\equiv +$ (or $\tau\equiv -$), this
  implies  that $T_{\rm mix}\le \exp({cL^\gep})$. The same bound holds
  when the boundary conditions are all $+$ on three sides and all $-$
  on the remaining one. The result, although still very far from the
  expected Lifshitz behavior $T_{\rm mix}=O(L^2)$, considerably improves
  upon the previous known estimates of the form $T_{\rm mix}\le \exp({c
    L^{\frac 12 + \gep}})$. The techniques are based on induction
  over length scales, combined with a judicious use of the so-called
  ``censoring inequality'' of Y. Peres and P. Winkler, which in a sense
  allows us to guide the dynamics to its equilibrium measure.
\end{abstract}

\maketitle

\thispagestyle{empty}

\section{Introduction, model and main results}
Glauber dynamics for classical spin systems has been extensively
studied in the last fifteen years from various perspectives and across
different areas like mathematical physics, probability theory and
theoretical computer science. A variety of techniques have been
introduced in order to analyze, on an increasing level of
sophistication, the typical time scales of the relaxation process to
the reversible Gibbs measure (see e.g.
\cite{cf:notefabio,cf:Peresetal} and the recent work on the cutoff
phenomenon for the mean field Ising model \cite{cf:Peresmeanfield}).
These techniques have in general proved to be quite successful in the
so-called one-phase region, corresponding to the case where the system
has a unique Gibbs state. When instead the thermodynamic parameters of
the system correspond to a point in the phase coexistence region, a
whole class of new dynamical phenomena appear (coarsening, phase
nucleation, motion of interfaces between different phases,...) whose
mathematical analysis at a microscopic level is still quite far from
being completed.

A good instance of the latter situation is represented by the Glauber
dynamics for the usual $\pm 1$ Ising model at low temperature in the
absence of an external magnetic field (see Section
\ref{sect:Glauber}). When the system is analyzed in a finite box of
side $L$ of the $d$-dimensional lattice $\bbZ^d$ with \emph{free}
boundary conditions, the relaxation to the Gibbs reversible measure
occurs on a time scale exponentially large in the surface $L^{d-1}$
\cite{cf:Thomas,cf:Sugimine2} because of the energy barrier between
the two stable phases of the system (see Section \ref{sect:main
  results} for a more quantitative statement). When instead one of the
two phases is selected by \emph{homogeneous} boundary conditions, e.g.
all pluses, then equilibration is believed to be much faster and it
should occur on a polynomial (in $L$) time scale because of the
shrinking of the big droplets of the opposite phase via motion by mean
curvature under the influence of the boundary conditions.
Unfortunately, establishing the above polynomial law in $\bbZ^d$
remains a kind of holy grail for the subject and the existing bounds
of the form $\exp({c\sqrt{L\log(L)}})$ in $d=2$
\cite{cf:M_2D,Higuchi-Wang} and $\exp({cL^{d-2}\log(L)^2})$ in $d\ge
3$ \cite{cf:Sugimine} are very far from it.

It is worth mentioning that, always for the low-temperature Ising
model but with the underlying graph $G$ different from $\bbZ^d$, it has
been possible to carry out a quite detailed mathematical analysis.
The first example is represented by the regular $d$-ary tree
\cite{cf:Martinelli-Sinclair-Weitz} and the second one by certain
hyperbolic graphs \cite{cf:Bianchi}. In both cases one can show for
example that the \emph{relaxation time} or inverse \emph{spectral gap}
of the Glauber dynamics in a finite ball with all plus boundary
conditions is uniformly bounded from above in the radius of the ball,
a phenomenon that is believed to occur also in $\bbZ^d$ in large
enough ($\ge 4$?) dimension $d$.

Moreover polynomial bounds on the mixing time, sometimes with optimal
results, have been proved for some simplified models of the random
evolution of the phase separation line between the plus and minus
phase for the two-dimensional Ising model (see for instance
\cite{cf:CMT} and \cite{cf:MartSinc}).  The latter contribution, in
particular, partly triggered the present work. There, in fact, the
opportunities offered by the so-called Peres-Winkler \emph{censoring
  inequality} \cite{cf:noteperes} have been detailed in the very
concrete and non-trivial case of the so-called \emph{Solid-on-Solid}
model.

Roughly speaking the censoring inequality (see Section
\ref{sect:censoring}) says that, when considering the Glauber dynamics
for a monotone system like the Ising model on a finite graph and under
certain conditions on the initial distribution, 
switching off (i.e., censoring) the spin flips in some part of the graph
and for a certain amount of time can only increase the variation
distance between the distribution of the chain at the final time $T$
and the equilibrium Gibbs measure. Therefore, if the censored dynamics
is close to equilibrium at a certain time $T$, the same holds for
the true (\ie uncensored) one.

The fact that the choice of where and when to implement the censoring
is completely arbitrary (provided that it is independent of the actual
evolution of the chain) offers the possibility of (sort of) guiding the
dynamics towards the stationary distribution through a sequence of
local equilibrations in suitably chosen subsets of the graph. Of course
the local equilibrium in each of the sub-graphs is conditioned to the
random configuration reached by the dynamics outside it and therefore
one is naturally led to consider the Ising model with \emph{random
  boundary conditions}, a quite delicate topic because of the extreme
sensitivity of the relaxation or mixing time to boundary conditions
(see \cite{cf:Alex,cf:AlexYosh, cf:HY,cf:SY} for several results in this
direction, some of them quite surprising at first sight). Moreover it
should also be clear that, in order for the guidance process to be
successful, the distribution of the random boundary conditions at each
stage of the censoring should be close to that provided by the
stationary Gibbs distribution, a requirement that puts quite severe
restrictions on the choice of the censoring scheduling.

The main contribution of this paper is a detailed implementation of
this program for the two-dimensional, low-temperature, Ising model in
a finite box with either homogeneous, \ie all plus (all minus),
boundary conditions or, more generally, random boundary conditions
that are stochastically larger (stochastically smaller) than those
distributed according to the plus (minus) phase.

In order to state precisely our results we need to define the model,
fix some useful notation and recall some basic facts about the Ising
model below the critical temperature.

\subsection{The standard Ising model}
Let $\L$ be a generic finite subset of $\bbZ^2$.
Each site $x$ in $\L$ indexes a spin $\s_x$ which takes values $\pm
1$.  The spin configurations $\{ \s_x \}_{x \in \L}$ have a
statistical weight determined by the Hamiltonian 
\begin{equation*} H^{\t} ( \s ) =
- \frac12\sum_{x,y \in \L \atop |x-y|=1} \s_x\s_y
- \sum_{x \in \L, y  \in \L^c \atop |x-y|=1}\s_x \t_y \, ,
\end{equation*}
where $\t = \{\t_y \}_{y\in \L^c}$ are boundary conditions
outside $\L$.

The Gibbs measure associated to the spin system with boundary conditions
$\t$ is \begin{equation*} \forall \s \in\O_\L:=\{-1,+1\}^\Lambda,
  \qquad \pi^{\t}_\L ( \s ) = \frac{1}{Z_{\b,\L}^{\t}}
  \exp \left( - \b H^{\t} ( \s ) \right),
\end{equation*}
where $\b$ is the inverse of the temperature ($\b = \frac{1}{T}$) and
$Z_{\b,\L}^{\t}$ is the partition function.  If the boundary
conditions are uniformly equal to $+1$ (resp. $-1$), then the Gibbs measure will
be denoted by $\pi_\L^+$ (resp. $\pi_\L^-$). 
If instead the boundary
conditions are free (i.e.  $\t_y=0 \ \forall y$) then the Gibbs
measure will be denoted by $\pi_\L^f$. 
\begin{remark}
 Sometimes we will drop the superscript $\t$ and
  the subscript $\L$ from the notation of the Gibbs measure.
\end{remark}
It is useful to recall a \emph{monotonicity property} of the Gibbs
measure that will play a key role in our analysis.  One introduces a
partial order on $\O_\L$ by saying that $\s \le \h$ if $\s_x\le\h_x$
for all $x\in \L$.  A function $f:\O_\L\mapsto\mathbb R$ is called
{\it monotone increasing (decreasing)\/} if $\s\le\h$ implies
$f(\s)\le f(\h)$ ($f(\s)\ge f(\h)$).  An event is called {\it
  increasing} ({\it decreasing}) if its characteristic function is
increasing (decreasing).  Given two probability measures $\mu$, $\nu$
on $\O_\L$ we write $\mu \preceq \nu$ if $\mu(f) \le\nu(f)$ for all
increasing functions $f$ (with $\mu(f)$ we denote the expectation of
$f$ with respect to $\mu$).  In the following we will take advantage
of the FKG inequalities \cite{cf:FKG} which state that
\begin{itemize}
\item if $\t\le\t'$, then $\pi_\L^\t \preceq \pi_\L^{\t'}$
\item if $f$ and $g$ are increasing then
   $\pi_\L^\t(f g) \ge \pi_\L^\t(f) \pi_\L^\t(g)$.
\end{itemize}

The phase transition regime occurs at low
temperature and it is characterized by spontaneous magnetization
in the thermodynamic limit.
There is a critical value $\b_c$ such that
\begin{equation}
\label{magnetization}
\forall \b > \b_c, \qquad
\lim_{\L\to \bbZ^2} \pi_\L^+ (\s_0) = -
\lim_{\L\to \bbZ^2} \pi_\L^- (\s_0) = m_\b >0 \, . 
\end{equation} 
Furthermore, in the thermodynamic limit the measures $\pi_\L^+$ and
$\pi_\L^-$ converge (weakly) to two distinct Gibbs measures $\pi^+_\infty$
and $\pi^-_\infty$ 
which are measures on the space $\Omega_{\Z^2}=\{ - 1,+1\}^{\bbZ^2}$.  Each of these
measures represents a pure state.

The next step is to quantify the coexistence of the two pure states
defined above. Let $\L_L= \{-\lfloor L/2\rfloor , \dots ,\lfloor
L/2\rfloor \}^2$, let $\vec{n}$ be a vector in the unit circle $\bbS$ 
and $\phi_{\vec n}$ the angle it forms with $\vec{e}_1=(1,0)$
and finally let $\t$ be the following mixed
boundary conditions
\begin{equation*} 
\forall y \in \L_L^c, \qquad \t_y  =
\begin{cases}
+1, & \qquad \text{if} \quad  \vec{n} \cdot y \geq 0, \\
-1, & \qquad \text{if} \quad  \vec{n} \cdot y <  0 . \end{cases}
\end{equation*} 
The partition function with mixed boundary conditions
is denoted by $Z_{\b, L}^{\pm} (\vec{n})$ and the one with boundary
conditions uniformly equal to $+1$ 
by $Z_{\b, L}^+$.
\begin{definition}
The surface tension in the direction orthogonal to $\vec{n} \in \bbS$
is an even and  periodic function of $\phi_{\vec n}$ of period $\pi/2$, and for
$-\pi/4\le \phi_{\vec n}\le \pi/4$ it
is defined by
\begin{equation}
\label{surfacetension}
\tau_\b(\vec{n}) =  \lim_{L\to \infty} \, - \frac{\cos(\phi_{\vec n})}{\beta L}
\,
\log \frac{Z_{\b, L}^{\pm} (\vec{n})}{Z_{\b, L}^+} \,
. \end{equation} 
\end{definition}
We refer to \cite{cf:MMR} for a general derivation of the
thermodynamic limit \eqref{surfacetension}.  With this definition, one
result (among many others) concerning the coexistence of the two
phases can be formulated as follows \cite{cf:Shlosman}.  Let
$m_{\L_L}(\s)=\sum_{x\in \L_L}\s_x$ be the total magnetization in the
box $\L_L$. Then
\begin{equation}
\label{eq:shlos}
\lim_{L\to \infty}-\frac{1}{L} \log\(\pi_{\L_L}^f(\inte{m_{\L_L}/2}=0)\)=\tau_\b 
\end{equation}
where $\t_\b$ is the surface tension in the horizontal direction $\vec
e_1$.

\subsection{ The Glauber dynamics}
\label{sect:Glauber}
The stochastic dynamics we want to study, sometimes referred to as the
\emph{heat-bath dynamics}, is a continuous time Markov chain on
$\O_\L$, reversible w.r.t. the measure $\pi_\L^\t$, that
can be described as follows. With rate one and for each vertex $x$,
the spin $\s_x$ is refreshed by sampling a new value from the set
$\{-1,+1\}$ according to the conditional Gibbs measure
$\pi_x:=\pi_\L^\t(\cdot \tc \s_y,\, y\neq x)$. It is easy to check
that the heat-bath chain is characterized by the generator
 \begin{equation}
    \label{eq:1}
(\cL_\L^{\t} f)(\s) = \sum_{x\in \L} \left[ \pi_{x} (f) - f(\s)
  \right]
 \end{equation}
where $\pi_{x}(f)$ denotes the average of $f$ with respect to  the conditional Gibbs measure $\pi_x$, which acts only on 
the variable $\s_x$.
The Dirichlet form associated to $\cL_\L^{\t}$ takes the form 
$$
\cE_\L^{\t}(f,f) = \sum_{x\in \L} \pi^{\t}_\L\bigl(\, \Var_x(f)\,\bigr)
$$
where $\Var_x(f)$ denotes the variance with respect to 
$ \pi_{x}$.

We will always denote by $\mu_t^\s$ the distribution of the chain at
time $t$ when the starting point is $\s$. If $\s$ is either
identically equal to $+1$ or $-1$ then we simply write $\mu_t^+$ or
$\mu_t^-$. The boundary conditions $\t$ are usually not explicitly
spelled out for lightness of  notation. Sometimes we  write
$\mu^\si_{\Lambda,t}$ when we wish to emphasize that we are looking at
the evolution for a system enclosed in the domain $\Lambda$.
 
The Glauber dynamics with the heat-bath updating rule satisfies a
particularly useful monotonicity property. It is possible to construct
on the same probability space (the one built from the independent
Poisson clocks attached to each vertex and from the independent coin
tosses associated to each ring) a Markov chain $\{\h_t^{\s,\t}\}_{t\ge
  0}$, $(\s,\t)\in \O_\L\times \O_{\L^c}$, such that
\begin{itemize}
\item for each $\t\in \O_{\L^c}$ and $ \s\in \O_\L$ the coordinate process
  $(\h_t^{\s,\t})_{t\ge 0}$ is a version of the Glauber chain started from $\s$
with
  boundary conditions $\t$; \item for any $t\ge 0$,
  $\h_t^{\s,\t}\le \h_t^{\s',\t'}$ whenever $\s\le \s'$ and
  $\t\le\t'$.

\end{itemize}

It is possible to extend the above definition of the generator $\cL_\L^{\t}$ directly to the whole lattice $\bbZ^2$ and get a well defined
  Markov process on $\O_{\bbZ^2}$ (see e.g. \cite{cf:Liggett}). The latter will be referred
  to as the infinite volume Glauber dynamics, with generator denoted by $\cL$.

  Two key quantities measure the speed of relaxation to equilibrium of
  the Glauber dynamics. The first one is the \emph{relaxation time}
  $T_{\rm relax}$.
\begin{definition}
  $T_{\rm relax}$ is the best constant $C$ in the Poincar\'e inequality
\begin{equation}
  \label{eq:3}
  \Var_\L^\t(f):=\Var_{\pi_\Lambda^\tau}(f)\le C \cE_\L^\t(f,f),\qquad \forall \; f:\O_\L\mapsto \bbR.
\end{equation}
\end{definition}
In particular, for any $f:\O_\L\mapsto \bbR$, it follows that
\begin{equation}
  \label{eq:5}
 \Var_\L^\t\(\nep{t\cL_\L^\t}f\)^{1/2}\le \nep{-t/T_{\rm relax}}\Var_\L^\t\(f\)^{1/2}.
\end{equation}
We will write $\gap:=\gap^\tau_\L$ for the inverse of $T_{\rm relax}$.

\medskip

Another relevant quantity is the \emph{mixing time} which is defined as follows.
Recall that the total variation distance between two measures $\mu,\nu$ on a finite probability space $\O$ is defined as
\begin{eqnarray}
  \|\mu-\nu\|:=\frac12\sum_{\si\in\O}|\mu(\si)-\nu(\si)|.
\end{eqnarray}
\begin{definition}
For any $\epsilon\in (0,1)$, we define 
\begin{equation}
  \label{eq:4}
  T_{\rm mix}(\epsilon):=\inf\{t>0: \sup_\s \|\mu_t^\s-\pi_\L^\t\|\le \epsilon\}.
\end{equation}
When $\epsilon=1/(2e)$ we will simply write $T_{\rm mix}$.
\end{definition}
With this definition it follows in particular that (see e.g.
\cite{cf:Peresetal})
\begin{equation}
\label{eq:*}
  \sup_\s \|\mu_t^\s-\pi_\L^\t\|\le \(2\epsilon\)^{\inte{t/T_{\rm mix}(\epsilon)}} \quad \forall t\ge 0.
\end{equation}
As it is well known (see e.g. \cite{cf:Peresetal}) the following bounds
between $T_{\rm relax}$ and $T_{\rm mix}$ hold:
\begin{equation}
  \label{eq:6}
  T_{\rm relax}\le T_{\rm mix}\le \log\(\frac{2e}{ \pi^*}\)
T_{\rm relax}
\end{equation}
where $\pi^*=\min_\s \pi_\L^\t(\s)$. Notice that $\pi^*\ge \nep{-c|\L|}$ for some
constant $c=c(\beta)$ and therefore the two quantities differ at most by
const$\times$volume.

Another definition we will often need is the following:
\begin{definition}
  Let $\mu,\nu$ be measures on $\O_\L$, let $\si\in\O_L$ and $V\subset
  \L$. Then, $\|\mu-\nu\|_V$ denotes the variation distance
  between the marginals of $\mu$ and $\nu$ on $\Omega_V$, and 
$\si_V$ the restriction of $\si$ to $V$.
\end{definition}

\subsection{Main results}
\label{sect:main results}
Our main result considerably improves upon the existing \emph{upper
  bound} on the mixing time (and therefore also on the relaxation
time) when $\L$ is a square box and the boundary conditions $\t$ are
homogeneous \ie either all plus or all minus. As a by-product we also
get a new bound on the time auto-correlation function of, e.g., the
spin at the origin for the infinite volume Glauber dynamics started
from the plus phase $\pi^+_\infty$.  Before stating the results we
quickly review what was known so far. In what follows $\L_L$ will
always be a $L\times L$ box.

When the boundary conditions are free, a simple bottleneck argument proves that  
\begin{equation*}
  T_{\rm relax} \ge \frac{1}{L^2}\( \pi_{\L_L}^f(\inte{m_{\L_L}/2}=0)\)^{-1}
\end{equation*}
so that (recall \eqref{eq:shlos}) 
\begin{equation*}
  \lim_{L\to \infty} \frac 1L \log(T_{\rm relax})\ge \t_\b.
\end{equation*}
In \cite{cf:M_2D} such a result was improved to an \emph{equality} for large enough values of $\b$ and in \cite{cf:CGMS} for any $\b>\b_c$.

Quite different is the situation for homogeneous boundary conditions, e.g. all plus, for
which the bottleneck between the two phases is removed by the boundary
conditions and the relaxation process should occur on a much shorter
time scale. In this case one expects a polynomial growth of both
$T_{\rm relax}$ and $T_{\rm mix}$ of the form
\begin{equation*}
  T_{\rm relax}\approx L,\qquad T_{\rm mix} \approx L^2.  
\end{equation*}
The reason behind the difference in the power of $L$ of the two
growths seems to be quite subtle and largely not yet understood at the
mathematical level.  The only rigorous results in this direction are
those obtained in \cite{Bodineau-Martinelli} where, apart from
logarithmic corrections, the appropriate lower bounds on $T_{\rm
  relax}$ and $ T_{\rm mix}$ have been established by means of quite subtle
test functions combined with the whole machinery of the Wulff
construction.

As far as upper bounds are concerned, they proved to be quite hard to
obtain and the available results are still quite poor.  In the case of
homogeneous boundary conditions it was first shown in \cite{cf:M_2D}
that, for $\b$ large enough and any $\gep>0$,
\begin{equation*}
  T_{\rm relax} \le \exp\left({c L^{1/2 + \gep}}\right)
\end{equation*}
for a suitable constant $c$ depending on $\gep$ and $\beta$. Later such a bound
was improved to $\exp({c\sqrt{L\log L} })$ in \cite{Higuchi-Wang}. When
the inverse temperature $\b$ is just above the critical value, the
only available result is much weaker (see \cite{cf:CGMS}) and of the
form
\begin{equation*}
  \lim_{L\to \infty}\frac 1L \log(T_{\rm relax})= 0.
\end{equation*}
Finally when $f(\s)=\s_0$ the above bounds combined with some simple
monotonicity arguments prove that, for any $\a>0$,
\begin{equation*}
  \Var^+_\infty\(\nep{t\cL}f\)\le c/t^\a
\end{equation*}
(where $\Var^+_\infty$ denotes the variance w.r.t. the plus phase
$\pi^+_\infty$) while the expected behavior is $O(\nep{-\sqrt{t}})$,
see \cite{cf:Fisher-Huse}.

We are now in a position to state our main results.

\begin{Theorem}
\label{th:quadrato}
Let $\beta$ be large enough and let $L$ belong to the sequence $\{2^n-1\}_{n\in\N}$. 
\begin{enumerate}
\item If the boundary conditions (b.c.) $\tau$ are sampled from a law $\bP$ which
either stochastically dominates the pure phase $\pi^+_\infty$ or is stochastically dominated by $\pi^-_\infty$ (see Section \ref{sect:b.c}), there exists
$c=c(\beta,\gep)$ (independent of $\bP$) such that 
\begin{eqnarray}
  \label{eq:maintmix00}
  \bE \|\mu^\pm_{t_L}-\pi^\tau\|\le \exp\left(-c L^{\gep^2/16}\right),
\end{eqnarray}
where $t_L=\exp(c L^\gep)$.
In particular,
\begin{eqnarray}
\label{eq:maintmix0}
  \bP\left(T_{\rm mix}\ge t_L\right)\le 
\exp\left(-c L^{\gep^2/16}\right).
\end{eqnarray}

\item The estimates \eqref{eq:maintmix00}-\eqref{eq:maintmix0} hold also if $\bP$ is stochastically dominated by $\pi^-_\infty$ on 
one side of $\L_L$, and stochastically dominates $\pi^+_\infty$ on the union of the other three sides. Similarly if the role of $+$ and $-$ is reversed.
\end{enumerate}

\end{Theorem}
The most natural consequence of the above result is
\begin{Corollary}
Let $\b$ be large enough and let $L$ belong to the sequence $\{2^n-1\}_{n\in\N}$. Consider the square $\Lambda_L$ with b.c. $\tau\equiv +$. For every
$\gep>0$ there exists $c=c(\beta,\gep)<\infty$ such that
\begin{eqnarray}
\label{eq:maintmix}
  T_{\rm mix} \le e^{c L^{\gep}}.
\end{eqnarray}
The same bound holds if the boundary conditions are $+$ on three sides
and $-$ on the remaining one. Similarly if $+$ is replaced by $-$.
\end{Corollary}

\begin{remark}\ 

  {\bf (i)} In the proof of Theorem \ref{th:quadrato} and of Corollary
  \ref{main2} below, we need at some point some key equilibrium estimates
  which are proved in the appendix via standard cluster expansion
  techniques for values of $\beta$ large enough. However, we expect
  those bounds to hold for every $\beta>\beta_c$.  Since this is the
  only part of the proof where the value of $\b$ comes into play, we
  expect Theorem \ref{th:quadrato} and Corollary \ref{main2} to hold
  for any $\b>\b_c$.  Let us also point out that, while we restrict
  for simplicity to the nearest-neighbor Ising model, we believe that
  our techniques can be generalized without conceptual difficulties to
  ferromagnetic Ising models with finite-range interactions. In
  particular, cluster expansion results for large $\beta$ are known to
  hold also in this more general situation.

{\bf (ii)} The restriction that $L$ belongs to the sequence $\{2^n-1\}_{n\in\N}$ is purely technical and it is a consequence of
the iterative procedure we use. It would not be difficult to 
eliminate this restriction by somewhat modifying our iteration below (see Remark \ref{scale} at the end of the proof of Theorem \ref{th:rec}),
but we have decided not to do this, in order to keep the presentation
as simple as possible.

{\bf (iii)} The above results have been stated for the heat-bath
dynamics but they actually apply to any other single site Glauber
dynamics (e.g. the Metropolis chain) with jump rates uniformly
positive (e.g. greater than $\d>0$) as can be seen via standard
comparison techniques \cite{cf:notefabio}. More precisely, if $\hat
T_{\rm mix}$ and $\hat T_{\rm relax}$ denote the mixing and relaxation
times of the new chain, then there exist constants
$c,c'$ depending on $\d,\b$ such that $\hat T_{\rm mix}\le c|\L| \hat T_{\rm relax}
\le c' |\L| T_{\rm relax}\le c'|\L| T_{\rm mix}$; the  results we are after then
follow since  $|\L|$ represents a polynomial correction which is irrelevant
in our case.

{\bf (iv)} Notice that in some sense our result \eqref{eq:maintmix0} is not so far from optimality. Indeed, consider
the distribution $\bP$ such that $\tau=+$ except for the boundary sites which are at distance at most $L^\gep$ from one of the corners
of the box, where $\tau$ is sampled from $\pi^+_\infty$. Clearly $\bP$ stochastically dominates 
$\pi^+_\infty$. Then, with $\bP$-probability $\exp(-c L^{\gep})$, $\tau=-$ around the corners and, thanks to the results of
\cite{cf:Alex}, $T_{mix}\ge \exp(c L^\gep)$.
\end{remark}
\subsection{Applications}

It is intuitive that if the b.c. are all $+$  (all $-$) and we start from the 
all $+$ (all $-$) configuration, equilibration will be much quicker.
Indeed, we have the following
\begin{Corollary}
\label{th:cor++}
Let $\beta$ be large enough and $\tau\equiv +$. For every $\gep>0$
there exists $c=c(\beta,\gep)>0$ such that
\begin{eqnarray}
\label{eq:cor++}
\lim_{L\to\infty} \|\mu^+_{t_1}-\pi^\tau\|=0,
\end{eqnarray}
where $t_1:=\exp(c(\log L)^\gep)$ . 
By a global spin flip the same results holds if $+$ is replaced by $-$.
\end{Corollary}
Finally, here is the result about the decay of time auto-correlations for
the infinite-volume dynamics in a pure phase:
\begin{Corollary}
\label{main2}
Let $\b$ be large, let $f(\s) =\s_0$ and let $\rho(t)\equiv
\Var^+_\infty\(\nep{t\cL}f\)$ be the time auto-correlation of the spin
at the origin in the plus phase $\pi^+_\infty$. Then for any $\gep >0$
there exists a constant $c=c(\b,\gep)$ such that
\begin{equation}
\label{eq:decorrel}
\rho(t)\le c \,\nep{-(1/c)(\log t)^{1/\gep}} .
\end{equation}
\end{Corollary}

\section{Auxiliary definitions and results}
In this section we collect some more detailed notation that will be
needed during the proof of the main results, together with certain
additional auxiliary results that will play a key role in our
analysis.
	
\subsection{Geometrical definitions}
\begin{figure}[htp]
\begin{center}
  \includegraphics[width=0.8\textwidth]{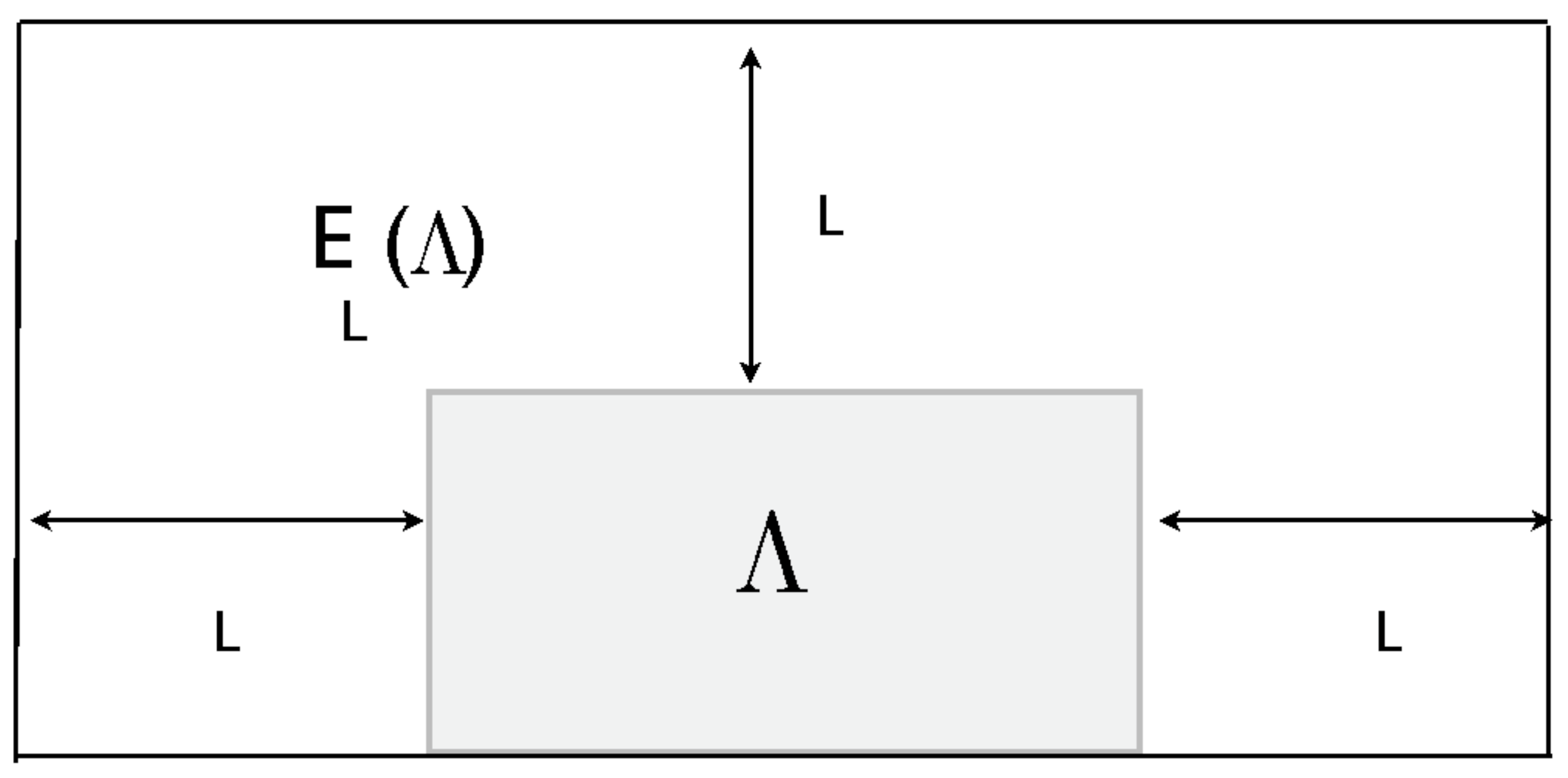}
\end{center}
\caption{The rectangle $\L$ and its
enlargement $E_L(\L)$}
\label{fig:enlarged} 
\end{figure}
The boundary of a finite subset $\L\subset \Z^2$, in the sequel
denoted  by $\partial \L$, consists of those sites in $\bbZ^2\setminus \L$
at unit distance from $\L$.  Given a rectangle $\L\subset \Z^2$ and
$L\in\N$, we denote by $ E_L(\L)$ the {\sl enlarged rectangle}
obtained from $\L$ by shifting by $L$ units the Northern boundary
upwards, the Eastern boundary eastward and the Western boundary
westward (see Figure \ref{fig:enlarged}).
 
Given $\gep>0$ (to be thought of as very small) and $L\in \N$ we let
\begin{equation*}
  R^\gep_L=\{(i,j)\in \bbZ^2:\ 1\le i\le L, \, 1\le j\le \lceil L^{\frac 12+\gep}\rceil\}.
\end{equation*}
Similarly we define the rectangle $Q^\gep_L$, the only difference being
that the vertical sides contain now $\lceil (2L+1)^{\frac 12+\gep}\rceil$
sites.  

\\
{\bf Notation warning.} In the sequel we will often remove the
superscript $\gep$ from our notation of the various rectangles
involved since it is a (small) parameter that we imagine given once
and for all.

\subsection{Boundary conditions}
\label{sect:b.c} A boundary condition $\tau$ for a given domain (typically, a rectangle) is an assignment of values $\pm1$ to each
spin on the boundary of the domain under consideration.

\begin{definition}
  A distribution $\bP$ of b.c. for a rectangle $R$ (which 
  will be $R_L$, $Q_L$ or a rectangle obtained by translating one of them
by a vector $v\in \Z^2$) is
  said to belong to $\mathcal D(R)$ if its marginal on the union of
  North, East and West borders of $R$ is stochastically dominated by
  (the marginal of) the minus phase $\pi^-_\infty$ of the infinite system,
  while the marginal on the South border of $R$ dominates the
  (marginal of the) infinite plus phase $\pi^+_\infty$.
\end{definition}
The most natural example is to take $\bP$ concentrated on the boundary
conditions $\tau$ given by $\tau\equiv -$ on the North, East and West
borders, and $\tau\equiv+$ on the South border. In that case we will
sometimes write $\pi_R^{-,-,+,-}$ for the equilibrium measure in $R$,
where we agree to order the sides of the border clockwise starting
from the Northern one.

\subsection{The inductive statements}
Here we define two inductive statements that will be proved later by a ``halving the scale'' technique.
\begin{definition}
For any given $L\in\N, \delta>0, t>0$ consider the system in $R_L$, with boundary condition $\tau$ chosen from 
some distribution $\bf P$. We say that
$\cA(L,t,\delta)$ holds if 
\begin{eqnarray}
  \bE\|\mu^{\pm}_t-\pi^\t\|\le \delta
\label{eq:A(L)}
\end{eqnarray}
for every $\bP\in \mathcal D(R_L)$. 

The statement $\cB(L,t,\delta)$ is defined similarly, the only difference
being that the rectangle $R_L$ is replaced by $Q_L$ (in particular,
$\bP$ is required to belong to $\mathcal D(Q_L)$).
\end{definition}

\subsection{Censoring inequalities}
\label{sect:censoring}

In this section, we consider the Glauber dynamics 
in a generic finite domain $\L\subset \Z^2$, not necessarily a rectangle.
The boundary conditions $\tau$ are not specified, because the results are
independent of it.

A fundamental role in our work is played by the censoring inequality
proved recently by Y. Peres and P. Winkler: this says, roughly speaking,
that removing (deterministically) some updates from the dynamics 
can only slow down equilibration, if the initial configuration is
the maximal (or minimal) one.

First of all we need a simple but useful lemma:
\begin{Lemma}\cite[Lemma 16.7]{cf:noteperes}
\label{th:lemma167}
Let $\pi,\mu,\nu$ be laws on a finite, partially ordered probability
space. If $\nu \preceq \mu$ and $\nu/\pi$ is increasing, i.e.
\begin{eqnarray}
  \frac{\nu(\s)}{\pi(\si)}\ge \frac{\nu(\h)}{\pi(\h)}
\end{eqnarray}
whenever $\si\ge \h$,
then
\begin{eqnarray}
  \|\nu-\pi\|\le \|\mu-\pi\|.
\end{eqnarray}
\end{Lemma}

The result of Peres-Winkler can be stated as follows:
\begin{Theorem}\cite[Theorem 16.5]{cf:noteperes}
\label{th:PW}
Let $m\in \N$, $\underline v:=(v_1,\ldots,v_m)$ a sequence of sites in
$\L$, and let $\underline v'$ be a sub-sequence of $\underline v$.  Let
$\mu_0$ be a law on $\O_{\L}$ such that $\mu_0/\pi$ is
increasing. Denote by
$\mu_{\underline v}$ the law obtained starting from $\mu_0$ and
performing heat-bath updates at the ordered sequence of sites
$\underline v$. Similarly for $\mu_{\underline v'}$. Then, 
\begin{eqnarray}
\label{eq:PW1}
  \|\mu_{\underline v}-\pi\|\le   \|\mu_{\underline v'}-\pi\|
\end{eqnarray}
and $ \mu_{\underline v}\preceq \mu_{\underline v'}$.  Moreover,
$\mu_{\underline v}/\pi$ and $\mu_{\underline v'}/\pi$ are increasing.
\end{Theorem}
It is easy to see that, if $\mu_0/\pi$ is instead
decreasing, \eqref{eq:PW1} still holds, while
the other statements become 
$ \mu_{\underline v'}\preceq \mu_{\underline v}$ and 
 $\mu_{\underline v}/\pi$, $\mu_{\underline v'}/\pi$ decreasing.

Here, ``performing a heat-bath update at a given site $v\in \Lambda$'' simply 
means freezing the configuration outside $v$ and extracting $\si_v$ from 
the equilibrium distribution conditioned on the configuration outside $v$.

Theorem \ref{th:PW} is proved in \cite{cf:noteperes} in the particular case
where $\mu_0$ is the measure concentrated at the all $+$ configuration,
but the proof of the above generalized statement  is essentially identical.
Let us emphasize that such result is not specific of the Ising model but
requires in an essential way monotonicity of the dynamics.

\medskip

From Lemma \ref{th:lemma167} and Theorem \ref{th:PW} we easily extract
the continuous-time censoring inequality we need:
\begin{Theorem}
  \label{th:PWgeneralizz}
  Let $n\in \N$, $0\equiv t_0<t_1<\ldots t_n\equiv T$ and $\L_i \subset
  \L, i=1,\ldots,n$.  Let $\mu_0$ be a law on $\O_{\Lambda}$
  such that $\mu_0/\pi$ is increasing.  Let $\mu_T$ be the law at time $T$ of the continuous-time,
  heat-bath dynamics in $\Lambda$, started from $\mu_0$ at time zero. Also,
  let $\mu'_T$ be the law at time $T$ of the modified dynamics which
  again starts from $\mu_0$ at time zero, and which is obtained from
  the above continuous time, heat-bath dynamics by keeping only the
  updates in $\L_i$ in the time interval $[t_{i-1},t_i)$ for
  $i=1,\ldots,n$. Then,
  \begin{eqnarray}
\label{eq:PWgen}
    \|\mu_T-\pi\|\le \|\mu'_T-\pi\|,
  \end{eqnarray}
and  $\mu_T\preceq \mu'_T$; moreover, $\frac{\mu_T}{\pi}$, $\frac{\mu'_T}{\pi}$ are 
both increasing.
\end{Theorem}
Needless to say, if instead $\mu_0/\pi$ is decreasing then all inequalities
except \eqref{eq:PWgen} are reversed.

{\sl Proof.} Let $m$ be the (random)
number of Poisson clocks which ring during the time interval $[0,T)$, and
denote by $s_i$ and $v_i\in \L, i\le m$ the times and sites where they
ring. We order the times as $s_i<s_{i+1}$ and of course $v_i$ are IID
and chosen uniformly in $\L$.  Define then
$w:=((v_1,s_1),\ldots,(v_m,s_m))$ and let $\mu_w$ be obtained from
$\mu_0$ performing single-site heat-bath updates at sites
$v_1,v_2,\ldots,v_m$ (in this order). Analogously, let $w'$ be
obtained by $w$ by removing all pairs $(v_j,s_j)$ such that $v_j\notin
\L_k$ where $k$ is such that $s_j\in [t_{k-1},t_k)$, and $\mu_{w'}$ be
defined in the obvious way.  For any realization of $w$ one has from
Theorem \ref{th:PW} that $\mu_w\preceq \mu_{w'}$ and that both
$\mu_w/\pi$ and $\mu_{w'}/\pi$ are increasing.  Since $\mu_T$ (respectively
$\mu'_T$) is just the average over $w$ of $\mu_w$ (resp. of $\mu_{w'}$),
one obtains all the claims of the theorem (except \eqref{eq:PWgen}) by
linearity. Inequality \eqref{eq:PWgen} comes simply from $\mu_T\preceq
\mu'_T$,  plus Lemma \ref{th:lemma167} and the fact that $\mu_T/\pi$
is increasing.
\qed

\medskip

We will need at various instances the following easy consequences
of the above facts.
\begin{Corollary}
\label{th:corollaPW}
  Let $t>0$ and assume that $\mu_0/\pi$ is increasing. 
Denote by $\mu_t$ the evolution started from $\mu_{t=0}=\mu_0$, and by
$\mu_t^+$ the one started from the maximal configuration $+$. Then
  \begin{eqnarray}
    \|\mu_t -\pi\|\le \|\mu_t^+-\pi\|.
  \end{eqnarray}
\end{Corollary} 
\begin{proof} 
We know
from Theorem \ref{th:PWgeneralizz} that $\mu_t/\pi$ is increasing.
Moreover, by monotonicity of the dynamics $\mu_t\preceq \mu_t^+$. The
claim then follows from Lemma \ref{th:lemma167}. 
\end{proof}

\begin{Corollary}
\label{th:corollaPW2}
Let $\g(t)=\max\(\|\mu^+_t-\pi\|,\|\mu^-_t-\pi\|\)$. Then
\begin{equation*}
  \g(t+s)\le 4\g(t)\g(s) \qquad \forall t,s\ge0.
\end{equation*}
\end{Corollary}
\begin{proof}
Notice that $\|\mu^+_{t+s}-\pi\|=\mu^+_{t+s}(A)-\pi(A)$ where $A=\{\s:\ \mu^+_{t+s}(\s)\ge \pi(\s)\}$. Because of  Theorem \ref{th:PWgeneralizz} the event $A$ is increasing so that $f:=\id_A-\pi(A)$ is an increasing function (and of course
$\pi(f)=0$). Thus 
\begin{eqnarray*}
\|\mu^+_{t+s}-\pi\|
&=& \mu^+_{t+s}(A)-\pi(A)\\
&=&\mu^+_{t}\(\mu^\s_s(f)\)\\
&=& \mu^+_{t}\(\mu_s^\s(f)\) -\pi\(\mu_s^\s(f)\)\\
&\le& 2\g(t)\sup_{\s}| \mu_s^\s(f)|\\
&\le & 2\g(t)\max\{| \mu_s^+(f)|,|\mu_s^-(f)|\}\\
&\le&4\g(t)\g(s) .
\end{eqnarray*}
Similarly for $\mu^-$. 
\end{proof}

\subsection{Perturbation of the boundary conditions and mixing time}
\label{sec:perturbation}
Consider a finite set $\L$ and two boundary conditions $\t,\hat\t$.
Let $T_{\rm mix} $ and $\hat T_{\rm mix}$ be the associated mixing times for
the Glauber chain in $\L$ with b.c. $\t$ and $\hat \t$, respectively.  Let
$M=\max\{\|\frac{\pi^\t}{\pi^{\hat \t}}\|_\infty
,\|\frac{\pi^{\hat\t}}{\pi^\t}\|_\infty\}$.
\begin{Lemma}There exists a constant $c$ independent of $\L,\t,\hat\t$ such that
  \begin{equation}
    \label{eq:7}
    T_{\rm mix} \le cM^3|\L|\hat T_{\rm mix}.
  \end{equation}
\end{Lemma}
\begin{proof}
Thanks to  \eqref{eq:6} and to the variational characterization of the relaxation time we get
  \begin{equation*}
    T_{\rm mix}\le c|\L|T_{\rm relax}\le c |\L|M^3 \hat T_{\rm relax} \le c |\L|M^3\hat T_{\rm mix} 
  \end{equation*}
where the third power of $M$ comes from expressing the Dirichlet form, the variance and the local variances w.r.t. $\pi^\t$ in terms of those w.r.t. $\pi^{\hat \t}$.
\end{proof}
Let now $\D\sset \partial \L$, let $\t_\D$ be some configuration in $\O_\D$, let $\bP$ be some distribution over the boundary conditions on $\partial \L$  and let $\bP^\D$ be the distribution which assigns probability zero to b.c. $\t$ not identically equal to $\t_\D$ on $\D$ and whose marginal on $\partial \L\setminus \D$ coincides with the same marginal of $\bP$.  Notice that we can sample from $\bP^\D$ by first sampling from $\bP$ and then changing (if necessary) to $\t_\D$ the spins of $\t$ in $\D$. If the pair so obtained is denoted by $(\t,\hat \t)$ then the corresponding constant $M$ satisfies $M\le M_\D:=\nep{8\b |\D|}$.

Let $d^{\pm}(t)=\|\mu_t^\pm-\pi^\t\|$ so that $\g(t)=\max\{d^+(t),d^-(t)\}$. Similarly for $\hat d^{\,\pm}(t),\hat \g(t)$. 
\begin{Lemma}
With the above notation
\begin{equation*}
  \bE\(\g(t)\)\le \nep{-M_\D}+ 8\bE\(\hat \g(\hat t)\)
\end{equation*}
where $\hat t = t/(c|\L|^2M_\D^4)$.
\end{Lemma}
\begin{proof}
Thanks to \eqref{eq:7}  and \eqref{eq:*},
\begin{gather*}
 \bE\(\g(t)\)\le \nep{-M_\D}+ \bP\(T_{\rm mix}\ge t/M_\D\)  \le \nep{-M_\D}+\bP\(\hat T_{\rm mix}\ge t/(c|\L|M_\D^4)\)\\
=\nep{-M_\D}+\bP\(\hat T_{\rm mix}\ge |\L|\hat t\).
  \end{gather*}
  Notice that, for any $s\ge 0$, $\hat T_{\rm mix}\ge s$ implies that
  there exists some starting configuration $\s$ for which the
  variation distance of its distribution at time $s$ from the
  equilibrium measure $\pi^{\hat\t}$, call it $\hat d^{\s}(s)$, is at
  least $1/(2\nep{})$. However, using the global monotone coupling of
  the Glauber chain,
\begin{gather}
  \label{eq:gathero} {\hat d}^{\,\s}(s)\le \bbP\(\h^{+,\,\hat
    \t}_s\neq \h_s^{-,\hat \t}\)\le
  \sum_{x\in \L}\bigl[\bbP(\h^{+,\hat \t}_s(x)=+)-\bbP(\h^{-,\,\hat \t}_s(x)=+)
\bigr]\\
  \le |\L|\(\hat d^{\,+}(s) + \hat d^{\,-}(s)\)\le 2|\L|\hat\g(s)
\end{gather}
 and therefore
\begin{equation*}
  \bP\(\hat T_{\rm mix}\ge |\L|\hat t\)\le \bP\(\hat \g(|\L|\hat t\,)\ge \frac{1}{4\nep{}|\L|}\).
\end{equation*}
Thanks to Corollary \ref{th:corollaPW2}, $\hat \g(t)\le \(4\hat\g(t_0)\)^{\inte{t/t_0}}$ so that
  \begin{gather*}
  \bP\(\hat \g(|\L|\hat t\,)\ge \frac{1}{4\nep{}|\L|}\)\le \bP\(\hat \g(\hat t\,)\ge \frac{1}{8}\)
\le 8\bE\(\hat \g(\hat t\,)\).
  \end{gather*}
\end{proof}

Let us remark for later convenience that, exactly like in  \eqref{eq:gathero},
 one proves that
\begin{eqnarray}
  \label{eq:supgamma}
  \sup_\si \|\mu_t^\si-\pi^\tau\|\le 2|\L|\gamma(t).
\end{eqnarray}

With the same notation the following will turn out to be quite useful:
\begin{Corollary}
 \label{th:corollaPW3} 
 Let $R_L\equiv R_L^\gep$ and let $\bP\in \cD(R_L)$. Let also $\D\subset 
\partial R_L$ be such that $L^{3\gep}\le |\Delta|\le 2L^{3\gep}$.
Assume that $\bE^\D\(\|\mu_t^\pm -\pi^\t\|\)\le \d$ for every $\bP\in
\cD(R_L)$. Then the statement $\cA(L, t',\d')$ holds true with
$\d'=8\d +\nep{-\nep{8\b L^{3\gep}}}$ and $t'=t\nep{cL^{3\gep}}$ for
some constant $c>0$ independent of $\Delta$ and $\tau_\Delta$.
Analogously $\cA(L, t,\d)$ implies $\bE^\D\(\|\mu_{t'}^\pm
-\pi^\t\|\)\le \d'$. Similar statements hold if we replace $R_L$ by
$Q_L$ and $\cA(L,t',\d')$ by $\cB(L,t',\d')$.
  
\end{Corollary}

\section{Recursion on scales: the heart of the proof}
This section represents the key of our results. We will inductively
prove over the sequence of length scales $L_n=2^{n+1}-1$ that the
statement $\cA(L_n,t_n,\d_n)$ and its analog $\cB(L_n,t_n,\d_n)$ hold true
for suitable $t_n,\d_n$ (see Theorem \ref{th:rec} below). In all this
section $\gep>0$ is fixed very small once and for all. Accordingly,
for any $L\in \bbN$, $R_L\equiv R_L^\gep$ and similarly for
$Q_L$. Finally $c,c'$ will denote positive numerical constants whose value may
change from line to line.

First we give a rough estimate which provides the starting point of the recursion:
\begin{Proposition}
\label{th:raf}
  For every $\beta$ there exists $c=c(\beta)$ such that for every 
$L\in\N$ 
the statements 
$
\cA(L,t,e^{-t\,e^{-c L}})
$
and
$
\cB(L,t,e^{-t\,e^{-c L}})
$
hold.
\end{Proposition}
\begin{proof} 
From rough estimates on the spectral gap \cite[Corollary 2.1]{cf:M_2D} and \eqref{eq:6}, one has that
\begin{eqnarray}
\label{eq:tmixraf}
T_{\rm mix}\le e^{cL}  
\end{eqnarray}
uniformly in the boundary conditions $\t$ and in $ L\in\N$, both for
$R_L$ and for $Q_L$. Applying \eqref{eq:*} with $\epsilon=1/(2e)$,  the claim is proved.
\end{proof}

\begin{Theorem}
\label{th:rec}
  For every $\beta$ there exist constants $c,c'$ such that:
\begin{enumerate}
\item if $\cA(L,t,\delta)$ holds, then also $\cB(L,2t,\delta_1)$ does,
with 
$$
\delta_1=\delta_1(L,\delta,t)=c\left(
\delta+\nep{-c' L^{2\gep}}+L^2\nep{-c' \log t}\right).
$$
\item If $\cB(L,t,\delta)$ holds, then also $\cA(2L+1,t_2,\delta_2)$ holds,
with 
\begin{eqnarray}
\label{eq:t2}
t_2=t_2(L,t)= \nep{cL^{3\gep}}t
\end{eqnarray}
and
\begin{eqnarray}
\delta_2=\delta_2(L,\delta)=c(\delta+\nep{-c'L^{3\gep}}).
\end{eqnarray}

\end{enumerate}
\end{Theorem}
Assuming the theorem we deduce the 
\begin{Corollary}
\label{th:corolla}
There exist $c,c'>0$ such that
the following holds.
For every $L\in\{2^{n}-1\}_{n\in\N}$
there exists
\begin{eqnarray}
\Delta(L)\le \exp\left(-c'L^{\gep^2}\right)
\end{eqnarray}
such that
$\cA\(L,t,\Delta(L)\)$ holds for every 
$t\ge e^{cL^{3\gep}}$.
\end{Corollary}
\begin{proof}[Proof.]
Note that if one iterates $j$ times the map $x\mapsto 
2x + 1$ starting from $x = 1$ one obtains 
$2^{j+1} -1=:L_j$. Assume now that $L=L_n$ for some large $n$ and set $n_0:=\lfloor \gep n\rfloor$, so that $(1/c)L^\gep\le L_{n_0}\le c L^\gep$. 
  
From  Theorem \ref{th:rec} one sees that it is possible to choose $c,c'>$ such that 
\begin{eqnarray}
\label{eq:implica}
\cA(L_j,t_j,\delta_j)\Longrightarrow \cA(L_{j+1},t_{j+1},\delta_{j+1})
\end{eqnarray}
 with
\begin{eqnarray}
t_{j+1}=2\,t_j\,\nep{c L_j^{3\gep}}
\end{eqnarray}
and
\begin{eqnarray}
\label{eq:deltaj+1}
\delta_{j+1}=c\left(\delta_j+\nep{-c' L_j^{2\gep}}+
L_j^2\,e^{-c'\log t_j}\right).
\end{eqnarray}
Let  
$$
t_{n_0}\equiv \nep{c L^{3\gep}}
$$
so that, thanks to Proposition \ref{th:raf},
$\cA(L_{n_0},t_{n_0},\delta_{n_0})$ holds with
\begin{eqnarray}
\label{eq:delta0}
\delta_{n_0}=\exp\({-\nep{c L^{3\gep}}}\).
\end{eqnarray}
Then, applying  \eqref{eq:implica} $n-n_0$ times, one obtains the claim 
$\cA(L,T(L),\Delta(L))$ with 
\begin{eqnarray}
T(L):= 2^{n-n_0}\nep{c\sum_{j=n_0}^{n}L_j^{3\gep}}\le
\nep{ c L^{3\gep}}
\end{eqnarray}
and
\begin{eqnarray}
  \Delta(L)\le L^{c}\left[
\delta(n_0)+\left(
\nep{-c' L_{n_0}^{2\gep}}+\,\nep{-c'\log (t_{n_0})}
\right)
\right]\le  \nep{-c L^{\gep^2}},
\end{eqnarray}
for a suitable constant $c$, where we used the rough bound (cf. \eqref{eq:deltaj+1})
\begin{eqnarray}
  \delta_{j+1}\le c\left(\delta_j+
\nep{-c'L_{n_0}^{2\gep}}+L^2\,\nep{-c'\log (t_{n_0})}\right).
\end{eqnarray}
The statement for every $t\ge T(L)$ then follows from Corollary \ref{th:corollaPW2}.
\end{proof}

\subsection{Proof of Theorem \ref{th:rec}: part (1)}
\label{sec:stat1}
\\
{\bf i)} We begin by proving that for every distribution $\bP\in \mathcal
D(Q_L)$ one has
\begin{eqnarray}
\bE\(\|\mu^+_{2t}-\pi^\t\|\)\le \delta_1.  
\end{eqnarray}
Observe that $Q_L$ can be seen as the union of two overlapping
rectangles $A$ and $B$, where $B$ is just the basic rectangle $R_L$
and  $A$  is obtained by shifting $B$ to the North by $\lceil
(2L+1)^{1/2+\gep}\rceil- \lceil L^{1/2+\gep}\rceil$ (see Figure
\ref{fig:Q}).
\begin{figure}[t]
\centering
\includegraphics[width=1.0\textwidth]{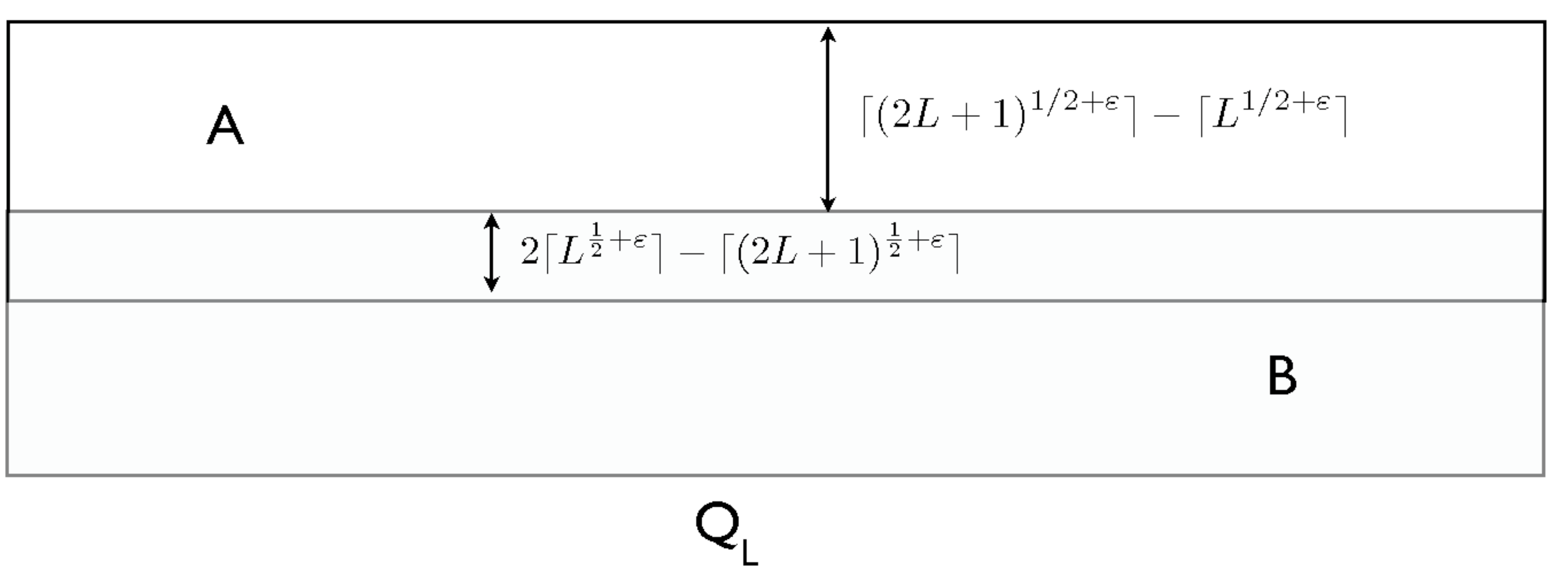}
\caption{$Q_L$ and its
covering with the rectangles $A,B$}
\label{fig:Q} 
\end{figure}

Let now $\tilde \mu_{2t}^+$ denote the distribution at time $2t$ of
the dynamics started from the all $+$ configuration and subject to the
following ``massage'': in the time interval $[0,t)$ we keep only the
updates in $A$, at time $t$ we increase all the spins in $B$ to $+1$
and in the interval $(t,2t]$ we keep only the updates in $B$.
\begin{Lemma}
\label{massage}
$$
\|\mu^+_{2t}-\pi^\t\|\le \|\tilde\mu^+_{2t}-\pi^\t\|
$$
\end{Lemma}
\begin{proof}
Let $\hat \mu^+_{2t}$ denote the distribution at time $2t$ of the dynamics started from the all 
$+$ configuration and subject to the following ``censoring'': in the time
interval $[0,t)$ we keep only the updates in $A$ and in the
interval $[t,2t]$ only the updates in $B$. By Theorem \ref{th:PWgeneralizz}, $\frac{\hat \mu^+_{2t}}{\pi^\t}$ is increasing. Moreover $\hat \mu^+_{2t}\preceq \tilde \mu^+_{2t}$ which combined with Lemma \ref{th:lemma167} proves  the result.   
\end{proof}
In order to better organize the notation we need the following: 
\begin{definition}
\label{def:35}
 We let 
\begin{enumerate}[(a)]
\item $\nu_1$ be the distribution obtained at time $t$ after the first
  half of the ``massage''. Clearly $\nu_1$ assigns zero probability to
  configurations that are not identical to $+$ in $A^c$;
\item $\nu_2^\s$ be the distribution obtained from the second half of
  the censoring starting (at time $t$) from a configuration equal to $+$ in $B$ and to $\s$ in $B^c$. Clearly $\nu_2^\s$ assigns zero probability to
  configurations that are not identical to $\s$ in $B^c$;
\item $\pi_A^{\t,+}:=\pi^\t(\cdot\tc \s_{A^c}=+)$;
\item $\pi_B^{\t,\h}:= \pi^\t(\cdot\tc \s_{B^c}=\h)$;
\item $\pi^{\t,-}$ (resp. $\pi^{\tau,+}$) be the Gibbs measure in
  $Q_L$ with minus (resp. plus) b.c. on its South boundary and $\t$ on
  the North, East and West borders.
\end{enumerate}
\end{definition}
With these notations the distribution $\tilde \mu^+_{2t}$ is written as 
\begin{equation*}
  \tilde\mu^+_{2t}(\h)= \nu_1(\h_{B^c})
\nu_2^{\h_{B^c}}(\h).
\end{equation*}
Notice that also the Gibbs measure $\pi^\t$ has a similar expression, namely,
\begin{equation*}
  \pi^\t(\h)= \pi^\t(\h_{B^c})\pi_B^{\t,\h_{B^c}}(\h).
\end{equation*}
Therefore
\begin{gather}
  \frac 12 \sum_\h | \tilde\mu^+_{2t}(\h)-\pi^\t(\h)|\nonumber \\
\le \frac 12 \sum_{\h}
\big |\nu_1(\h_{B^c})-\pi_A^{\t,+}(\h_{B^c})\big|\nu_2^{\h_{B^c}}(\h)
+  \frac 12 \sum_{\h}\big |
\pi_A^{\t,+}(\h_{B^c})\nu_2^{\h_{B^c}}(\h) -\pi^\t(\h)\big|
\nonumber \\
= 
\|\nu_1-\pi_A^{\t,+}\|_{B^c} + \|\g-\pi\|
\label{fm1}
\end{gather}
where 
\begin{equation*}
  \g(\h):=\pi_A^{\t,+}(\h_{B^c})\nu_2^{\h_{B^c}}(\h).
\end{equation*}
Clearly 
\begin{equation*}
  \|\g-\pi\| \le \pi^{\t,-}\bigl(\|\nu_2^{\h_{B^c}} - \pi_B^{\t,\h_{B^c}}\|\bigr) 
+ \|\pi_A^{\t,+}-\pi^\t \|_{B^c} +\|\pi^\t-\pi^{\t,-}\|_{B^c}.
\end{equation*}
In conclusion
\begin{gather}
  \bE\(\|\mu_{2t}^+-\pi^\t\|\)\le \bE\(\|\nu_1-\pi_A^{\t,+}\|_{B^c}\) \nonumber \\
  +\bE\(\pi^{\t,-}\bigl(\|\nu_2^{\h_{B^c}} - \pi_B^{\t,\h_{B^c}}\|\bigr)\)
  +\bE\(\|\pi_A^{\t,+}-\pi^\t \|_{B^c}\) + \bE\(\|\pi^\t-\pi^{\t,-}\|_{B^c}\).
  \label{FM0}
\end{gather}
By assumption the first term in the r.h.s. of \eqref{FM0} is smaller than $\d$. 
Next we analyze the second term. 
In this case, if we denote the four
boundary conditions around $B$, ordered clockwise starting from the
North one, by $\t_1,\t_2,\t_3,\t_4$, then their distribution $\bP^-$ is
given by
$$
\bP^-(\t_1,\t_2,\t_3,\t_4)= \bP(\t_2,\t_3,\t_4) \bE\(\pi^{\t,-}(\t_1)\tc \t_2,\t_4\).
$$ 
Notice that the marginal of $\bP^-$ on $\t_3$ coincides with that of
$\bP$ and therefore stochastically dominates the corresponding marginal
of $\pi^+_\infty$. It remains to examine the marginal on $(\t_1,\t_2,\t_4)$.
Let $f$ be a \emph{decreasing} function of these variables and observe
that, as a function of the boundary conditions on the North, East and
West sides of $Q_L$, the average $\pi^{\t,-}(f)$ is also
decreasing. Therefore, since $\bP\in \cD(Q_L)$,
\begin{equation}
 \bE^-(f)= \bE\(\pi^{\t,-}(f)\) \ge \pi^-_\infty\(\pi^{\t,-}(f)\)\ge \pi^-_\infty(f)
\label{FM2bis}
\end{equation}
\ie $\bP^-\in \cD(B)$.  Therefore 
\begin{equation*}
  \bE\(\pi^{\t,-}\bigl(\|\nu_2^{\h_{B^c}} - \pi_B^{\t,\h_{B^c}}\|\bigr)\)=\bE^-\(\|\nu_2^{\h_{B^c}} - \pi_B^{\t,\h_{B^c}}\|\)\le \d.
\end{equation*}
The third and the fourth term in \eqref{FM0} can be bounded from above by
essentially the same argument which we now present only for the fourth
term.
Clearly, for any choice of the boundary conditions $\t$, $\pi^{\t,-}\preceq \pi^\t$. Therefore
\begin{equation*}
 \bE\( \|\pi^\t-\pi^{\t,-}\|_{B^c}\) \le \sum_{x\in B^c} \bE\(\pi^\t(\s_x=+)-\pi^{\t,-}(\s_x=+)\).
\end{equation*}
\begin{claim}
\label{claim1}
There exists $c=c(\beta,\gep)>0$ such that 
\begin{eqnarray}
\label{eq:gamma3}
  \bE\(\pi^\t(\s_x=+)-\pi^{\t,-}(\s_x=+)\)\le \nep{-cL^{2\gep}}
\end{eqnarray}
for every $x\in B^c$.
\end{claim}
\begin{proof}[Proof]
  Let $\Gamma$ denote the event that in $B$ there is a $*$-connected
  chain (\ie either the Euclidean distance between two consecutive vertices $v,v'$ of the chain equals $1$, or it equals $ \sqrt{2}$
and in that case the segment $vv'$ forms an angle $\pi/4$ with the
horizontal axis) of $-$ spins which connects the East and West sides of
  $B$. By monotonicity,
\begin{eqnarray}
\label{eq:gamma2}
  \pi^\t(\sigma_x=+\tc \G) \le \pi^{\t,-}(\s_x=+)
\end{eqnarray}
and therefore 
\begin{equation*}
  \pi^\t(\sigma_x=+) - \pi^{\t,-}(\s_x=+)\le \pi^\t(\G^c).
\end{equation*}
By monotonicity 
\begin{equation*}
  \bE\(\pi^\t(\s_x=+)-\pi^{\t,-}(\s_x=+)\)\le \bE \pi^\t(\Gamma^c)\le \pi^-_\infty\(\pi^{\t,+}(\Gamma^c)\)
\end{equation*}
where we recall that the superscript $+$ means that on the South
border of $Q_L$ the b.c. are all plus.
\begin{figure}[h]
\begin{center}
\includegraphics[width=.6\textheight]{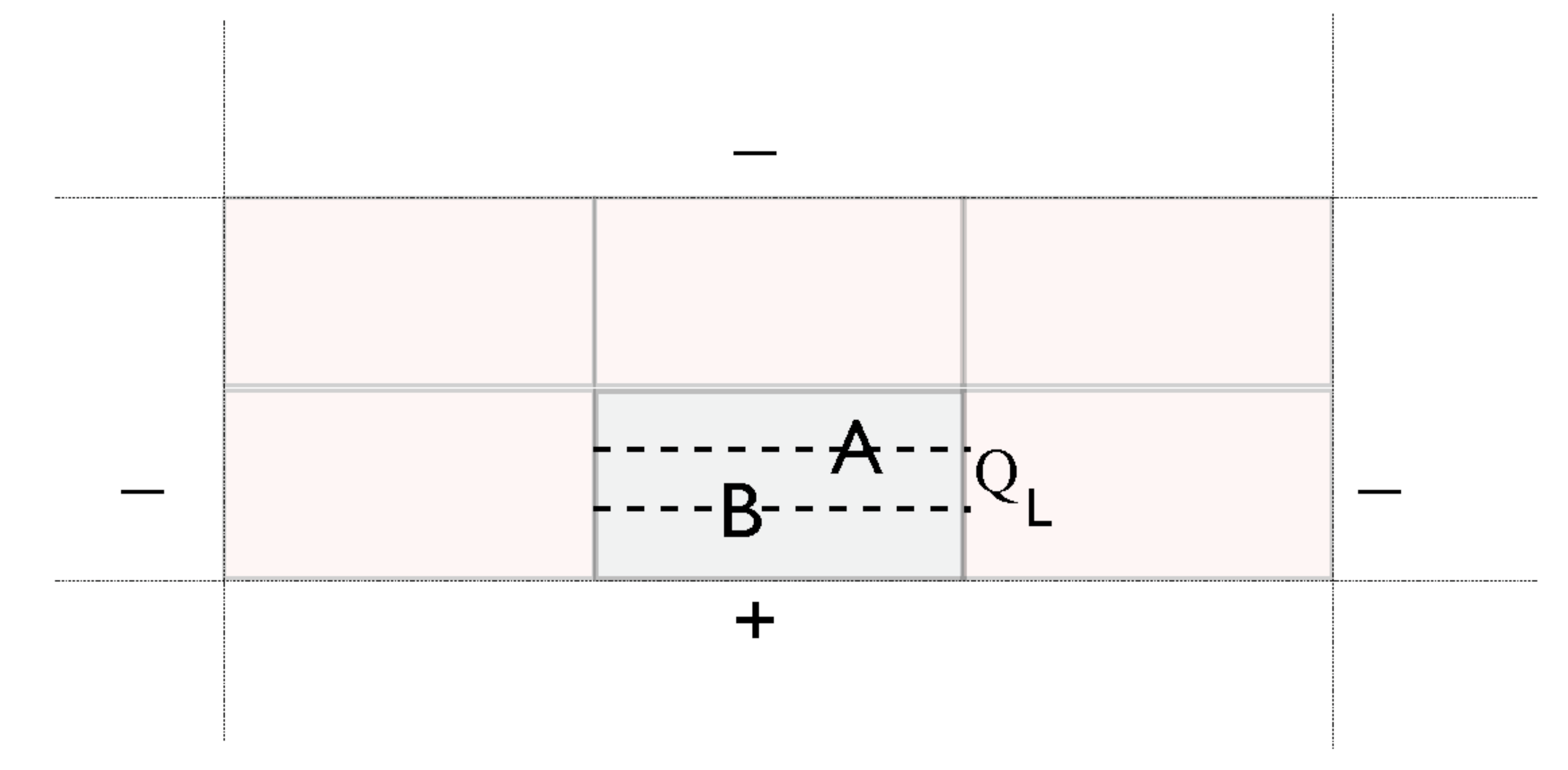}
\end{center}
\caption{The rectangle $Q_L$ (thick line) and its
enlargement $E_L(Q_L)$ (narrow line), with the b.c. of $\pi_\infty^{(-,-,+)}$}
\label{fig:enlargement} 
\end{figure}
Let $\pi_\infty^{(-,-)}$ be the the minus phase measure
$\pi^-_\infty$ conditioned to have all minuses on the 
North, East and West borders of the enlarged rectangle $E_L(Q_L)$
(see Figure \ref{fig:enlargement}). Standard bounds on the exponential decay of
correlations in the minus phase (see for instance
\cite{cf:Martinlof} or \cite[Chapter V.8]{cf:Simon})
prove that
\begin{equation}
\label{eq:standbo}
  \pi^-_\infty\(\pi^{\t,+}(\Gamma^c)\)\le \pi_\infty^{(-,-)}\(\pi^{\t,+}(\Gamma^c)\) + \nep{-cL}
\end{equation}
for some constant $c>0$. If we now add extra plus b.c. on the whole horizontal
line containing the South boundary of $Q_L$ and denote by $\pi_\infty^{(-,-,+)}$
the corresponding Gibbs measure then, by monotonicity and DLR equations,
we obtain
\begin{equation}
\label{eq:standbo2}
  \pi_\infty^{(-,-)}\(\pi^{\t,+}(\Gamma^c)\) \le \pi_\infty^{(-,-,+)}\(\pi^{\t,+}(\Gamma^c)\) =\pi^{(-,-,+)}_\infty\(\Gamma^c\). 
\end{equation}
Notice that $\pi_\infty^{(-,-,+)}$ is nothing but the Gibbs measure $\pi^{-,-,+,-}_{E_L(Q_L)}$ in the
rectangle $E_L(Q_L)$ of Figure \ref{fig:enlargement},
%
with $+$ b.c on the South border and $-$
b.c. on the rest of the boundary. 

Next, note that the event $\Gamma^c$ implies that 
the  unique open Peierls contour $\g$ (see definition in 
Appendix \ref{sec:equilibrio}) 
crosses the horizontal line containing the South border of $A$, and we will
prove in Appendix \ref{sec:primaeq} that
\begin{eqnarray}
  \label{eq:primaeq}
\pi^{-,-,+,-}_{E_L(Q_L)}\(\gamma \mbox{ reaches the height of the South border of  } A\) \le\nep{-c L^{2\gep}}.
\end{eqnarray}
The intuition for \eqref{eq:primaeq} is that the open contour $\g$
behaves like a one-dimensional simple random walk starting at the origin and
conditioned to stay positive and to return at time $L$ to the origin:
the probability that before this time it goes at distance of order
$L^{1/2+\gep}$ from the origin is smaller than $\exp(-cL^{2\gep})$.
\end{proof}

Altogether we have obtained 
\begin{eqnarray*}
  \bE\|\mu_{2t}^+-\pi^\tau\|\le 2\delta+e^{-cL^{2\gep}}.
\end{eqnarray*}
\\
{\bf ii)} Now we consider the dynamics started from the all $-$ configuration
and we prove 
\begin{eqnarray}
  \label{eq:distdameno}
  \bE\|\mu^-_{2t}-\pi^\t\|\le \delta_1.  
\end{eqnarray}
By Theorem \ref{th:PWgeneralizz}, $\|\mu^-_{2t}-\pi^\t\|\le
\|\tilde\mu^-_{2t}-\pi^\t\|$ where this time $\tilde \mu^-_{2t}$
denotes the distribution at time $2t$ obtained by starting the Glauber
dynamics from the minus initial condition and performing the following
``massage'' (the reverse of the previous one): in the time interval
$[0,t)$ we keep only the updates in $B$, at time $t$ we reset to $-$
all the spins in $A$ and in the time interval $[t,2t]$ we keep only
the updates in $A$.  In order to keep the notation as close as
possible to that of the previous case where the starting configuration
was all pluses we redefine
\begin{definition}\ 
  \label{def:37}
  \begin{enumerate}[(a)]
  \item $\pi_B^{\t,-}=\pi^\t(\cdot \tc \s_{B^c}=-)$;
  \item $\pi_A^{\t,\h}=\pi^\t(\cdot \tc \s_{A^c}=\h)$;
  \item $\nu_1$ is the distribution obtained after time $t$ and $\nu_2^{\s}$ is that obtained in the second time lag $t$  starting from the configuration equal to $-$ in $A$ and to $\s$ in $A^c$.    
  \end{enumerate}
\end{definition}
With these notations the same computation leading to \eqref{FM0} gives
\begin{equation}
  \label{FM5}
  \bE \|\tilde\mu^-_{2t}-\pi^\t\| \le
  \bE\|\nu_1-\pi_B^{\t, -}\|_{A^c} + \bE\pi^\t\bigl(\|\nu_2^{\h_{{A^c}}} - \pi_A^{\t,\h_{A^c}}\|\bigr) 
  + \bE\|\pi_B^{\t,-}-\pi^\tau \|_{A^c}.
\end{equation}
The first and third in the r.h.s of \eqref{FM5} are smaller than $\d$
and $\nep{-cL^{2\gep}}$ respectively by essentially the same arguments as before.
It remains to analyze the second term. Notice that 
\begin{gather*}
  \pi^\t\bigl(\|\nu_2^{\h_{{A^c}}} - \pi_A^{\t,\h_{A^c}}\|\bigr)\le
  \sum_{x\in A}\pi^\t\left[\nu_2^{\h_{A^c}}(\s_x=-)-
\pi^\t\(\pi_A^{\t,\h_{A^c}}(\s_x=-)\)\right]\\
  =\sum_{x\in A}\pi^\t\left[\nu_2^{\h_{A^c}}(\s_x=-)-\pi^\t(\s_x=-)\right].
\end{gather*}
Given $x\in A$ and $\ell\in\N$, let $K_\ell$ be the 
intersection of $A$ with a square of side $2\ell+1$, centered 
at $x$. Monotonicity implies that
\begin{eqnarray}
  \nu_2^{\h_{A^c}}(\s_x=-)\le \nu_{2,\ell}^{\h_{A^c}}(\si_x=-),
\end{eqnarray}
where $\nu_{2,\ell}^{\h_{A^c}}$ denotes the distribution at time $t$
obtained by the dynamics in $K_\ell$, started from all $-$, and with
b.c. which are all $-$ except on $\partial K_\ell\cap \partial A$
where the b.c. remain either $\tau$ (on the North, East and West
border of $A$) or $\h_{{A^c}}$ (on the South border of $A$).  Let
$\pi_{\ell}^{\t,\h_{A^c}}$ be the equilibrium measure of this
restricted dynamics. Then,
\begin{gather*}
\nu_2^{\h_{A^c}}(\s_x=-)-\pi^\t(\s_x=-)\nonumber \\
\le
 \left[ \nu_{2,\ell}^{\h_{A^c}}(\si_x=-) -\pi_{\ell}^{\t,\h_{A^c}}(\si_x=-)\right] 
+\left[\pi_{\ell}^{\t,\h_{A^c}}(\si_x=-)-\pi^\t(\si_x=-)\right]
  \\\nonumber
  \le e^{-t e^{-c \ell}}+
\left[\pi_{\ell}^{\t,\h_{A^c}}(\si_x=-)-\pi^\t(\si_x=-)\right],
\end{gather*}
where in the last inequality we used \eqref{eq:tmixraf}. If we now average first with respect to  $\pi^\t$ and then with respect to $\bP$ we claim that
\begin{claim}
\label{claim2}
On has for some $c>0$
\begin{gather}
\label{eq:dadim1}
\bE\(\pi^\t \bigl(\pi_{\ell}^{\t,\h_{A^c}}(\si_x=-)\bigr) - \pi^\t(\si_x=-)\)  \\
= \bE\(\pi^\t \bigl[\pi_{\ell}^{\t,\h_{A^c}}(\si_x=-) -
\pi_A^{\t,\h_{{A^c}}}(\si_x=-)\bigl]\) \le e^{-c\ell}.
\end{gather}
\end{claim}
(It is clear that if $\ell$ is so large that $K_\ell=A$, then
$\pi_\ell^{\tau,\eta_{A^c}}=\pi_A^{\tau,\eta_{A^c}}$ and the left-hand
side of \eqref{eq:dadim1} equals $0$).

Assuming the claim it is now sufficient to 
choose $\ell=\lceil(1/c)(\log t-\log \log t)\rceil$ to conclude that
\begin{eqnarray}
\bE\(\pi^\t\bigl(\|\nu_2^{\h_{{A^c}}} - \pi_A^{\t,\h_{A^c}}\|\bigr)\)
\le L^2e^{-c' \log t}
\end{eqnarray}
for some $c'>0$. 
\qed
\begin{proof}[Proof of Claim \ref{claim2}]
  Let $\G$ be the event that $x$ is separated from $\partial K_\ell\cap
  A$ by a $*$-connected chain of  minus spins. By monotonicity, for any $\h_{{A^c}}$,
\begin{equation*}
  \pi_A^{\t,\h_{{A^c}}}(\si_x=-\tc \G)\ge \pi_{ \ell}^{\tau,\h_{{A^c}}}(\si_x=-)
\end{equation*}
and therefore it is enough to show that
\begin{equation*}
  \bE\(\pi^\t\bigl(\pi_A^{\t,\h_{{A^c}}}(\G^c)\bigr)\)=
\bE\(\pi^\t(\G^c)\)\le \nep{-c\ell}.
\end{equation*}
The rest of the proof is now very similar to that of Claim
\ref{claim1}. Apart from an error $\nep{-cL}$ we can replace
$\bE\(\pi^\t(\G^c)\)$ by $\pi^{-,-,+,-}_{E_L(Q_L)}(\G^c)$, where $\pi^{-,-,+,-}_{E_L(Q_L)}$ is
the Gibbs measure on the enlargement $E_L(Q_L)$ (see again Figure \ref{fig:enlargement} above) with plus
b.c. on the South border and minus b.c elsewhere. In turn, thanks to the
fact that the event $\G^c$ depends only on the spins in $A$, we can
replace $\pi^{-,-,+,-}_{E_L(Q_L)}$ by the Gibbs measure $\pi_{E_L(Q_L)}^{-}$ on the same
region but with homogeneous \emph{minus} b.c. by paying an error smaller
than $\nep{-cL^{2\gep}}$. Finally, again by monotonicity and standard
correlations decay bounds in the pure phase,
\begin{equation*}
  \pi_{E_L(Q_L)}^{-}(\G^c)\le \pi_\infty^-(\G^c)\le \nep{-c\ell}
\end{equation*}
for some $c>0$.
\end{proof}

\subsection{Proof of Theorem \ref{th:rec} part (2)}
\label{sec:stat2}

Thanks to Corollary \ref{th:corollaPW3} and apart from the harmless
rescaling $t\mapsto t'=\nep{cL^{3\gep}}t$ and $\d \mapsto
\d'=c'\d+\nep{-cL^{3\gep}}$ for some constants $c,c'>0$, we can safely
replace the distribution $\bP$ over the boundary conditions outside
$R_{2L+1}$ with the modified distribution $\bP^\D$ (defined in Section \ref{sec:perturbation}), where
$\D=\{(i,0)\in \partial R_{2L+1}; |i-L|\le L^{3\gep}\}$ and the
pinned configuration $\t_\D$ is identically equal to $-1$. In other
words it is enough to prove that $\bE^\D\(\|\mu_{2t'}^\pm-\pi^\t\|\)\le
c\d'$.

\\
{\bf i)} As before we begin with the case where the dynamics in $R_{2L+1} $ is started from all pluses.  
Let now (see Figure \ref{fig:R})
\begin{eqnarray*}
A&=&Q_L+(\inte{L/2},0)\\
B&=&\{Q_L\}\cup\{Q_L+(L+1,0)\}\\ 
C&=&\{(i,j)\in R_{2L+1};\ i=L+1\}.  
\end{eqnarray*}
so that $R_{2L+1}=B\cup C$ and $B\cap C=\emptyset$.
\begin{figure}[th]
\centering
\includegraphics[width=0.9\textwidth]{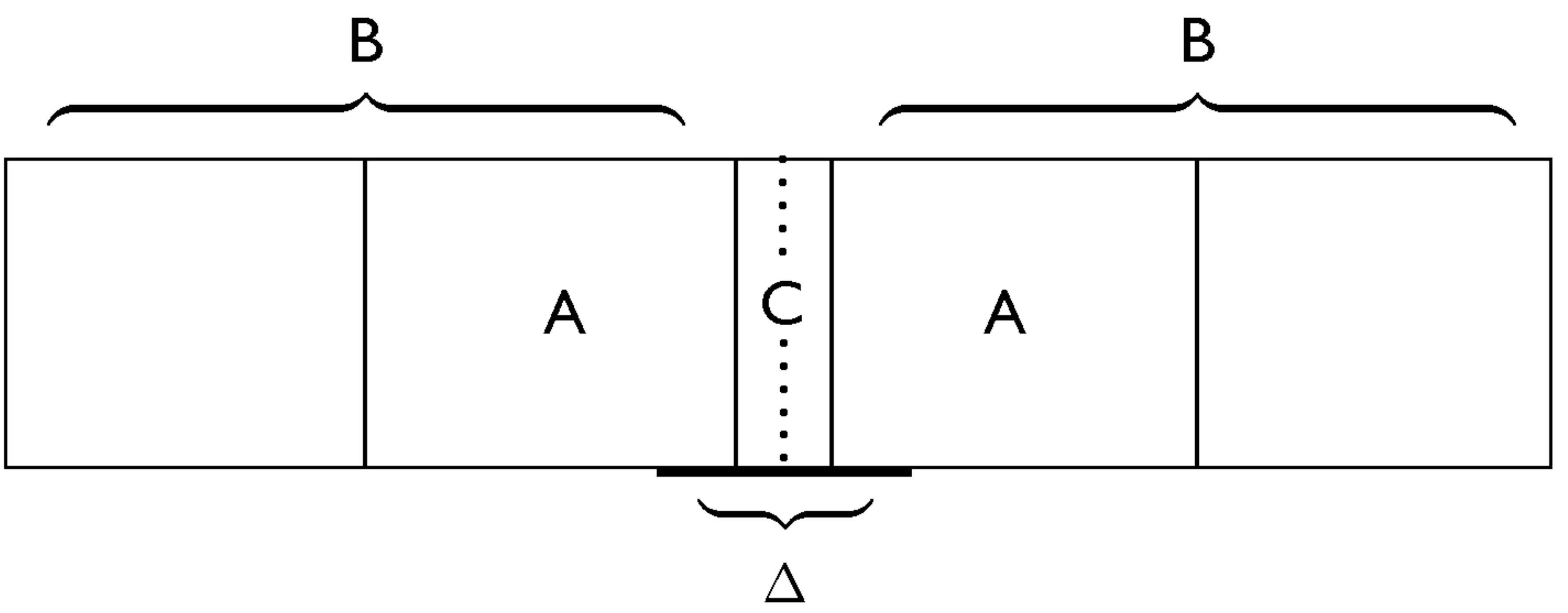}
\caption{$R_{2L+1}$ and its
covering with $A,B,C$. In bold the set $\Delta$
}
\label{fig:R} 
\end{figure}

By Theorem \ref{th:PWgeneralizz}, $\|\mu^+_{2t'}-\pi^\t\|\le
\|\tilde\mu^+_{2t'}-\pi^\t\|$ where, as before, the tilde indicates
that the following ``massage''has been applied: in the time interval
$[0,t')$ we keep only the updates in $A$, at time $t'$ we increase to
$+1$ all the spins in $B$ and in the interval $(t',2t']$ we keep only the
updates in $B$. Notice that the dynamics in $B$ in the time lag
$(t',2t']$ is a just a product dynamics in the two copies of $Q_L$, in
the sequel denoted by $B_1$ and $B_2$, whose union is $B$, with
boundary conditions $\t$ on $\partial B\cap \partial R_{2L+1}$ and
some boundary conditions on $C$ generated by the dynamics in $A$ in
the first time lag $[0,t']$.

\begin{definition} We define
\begin{enumerate}[(a)]
\item $\nu_1$ as the distribution obtained at time $t'$ after the first
  half of the censoring;
\item $\nu_2^\s$ as the distribution obtained from the second half of
  the censoring starting (at time $t'$) from a configuration equal to $\s$ in $C$ and to $+$ in $B$. Clearly $\nu_2^\s$ assigns zero probability to
  configurations that are not identical to $\s$ in $C$;
\item $\pi_A^{\t,+}:=\pi^\t(\cdot\tc \s_{A^c}=+)$ and similarly with $+$ replaced by $-$;
\item $\pi_B^{\tau,\eta_C}:=\pi^\tau(\cdot|\si_C=\eta_C)$;
\item $\pi^{\t,-}$ (resp. $\pi^{\tau,+}$) as the Gibbs measure in $R_{2L+1}$ with minus (resp. plus) b.c. on its South boundary and $\t$ on the North, East and West borders.
\end{enumerate}
\end{definition}
By proceeding exactly as in the proof of statement {\bf (1)} we get
\begin{gather}
\|\mu_{2t'}^+-\pi^\t\|\le \|\tilde \mu_{2t'}^+-\pi^\t\|\\\le
\|\nu_1-\pi_A^{\t,+}\|_{C} + \pi^{\t,-}\bigl(\|\nu_2^{\h_{C}} - \pi_B^{\t,\h_{C}}\|\bigr) 
+ \|\pi_A^{\t,+}-\pi^\t \|_{C} + \|\pi^{\t,-}-\pi^\t \|_{C} 
\label{FM6}
\end{gather}
and 
\begin{gather}
  \bE^\D\(\|\mu_{2t'}^+-\pi^\t\|\)\le \bE^\D\(\|\nu_1-\pi_A^{\t,+}\|_{C}\) \nonumber \\
  +\bE^\D\(\pi^{\t,-}\bigl(\|\nu_2^{\h_{C}} - \pi_B^{\t,\h_{C}}\|\bigr)\)
  +\bE^\D\(\|\pi_A^{\t,+}-\pi^\t \|_{C}\) + \bE^\D\(\|\pi^\t-\pi^{\t,-}\|_{C}\).
  \label{FM7}
\end{gather}
By assumption and thanks to Corollary \ref{th:corollaPW3}, if we
perform a global spin flip we see that the first term in the r.h.s. of
\eqref{FM7} is smaller than $\delta'$.  As far as the second term is
concerned we observe that the distribution $\bP^{\D,-}$ of the
boundary conditions $(\t,\h_C)$ given by
$\bP^{\D,-}(\t,\h_C)=\bP^\D(\t)\pi^{\t,-}(\h_C)$ coincides with the
$\D$-modification $(\bP^-)^\D$ of
$\bP^-(\t,\h_C)=\bP(\t)\pi^{\t,-}(\h_C)$. The same argument as in
\eqref{FM2bis} shows that the latter belongs to $\cD(B_i)$, $i=1,2$,
so that (via Corollary \ref{th:corollaPW3} and the immediate inequality
$\|\mu\otimes \nu-\mu'\otimes \nu'\|\le \|\mu-\mu'\|+\|\nu-\nu'\|$) the
second term is smaller than $2\d'$.

We now turn to the more delicate third and fourth term in the r.h.s. of \eqref{FM7}. Since they can be treated essentially in the same way we discuss only the third one. As usual we write
\begin{equation}
  \label{FM8}
  \bE^\D\(\|\pi_A^{\t,+}-\pi^\t \|_{C}\) \le \sum_{x\in C} \bE^\D\(\pi_A^{\t,+}(\s_x=+)-\pi^\t(\s_x=+) \).
\end{equation}
Let $\G$ be the event that in $A$ there exist two $*$-connected chains
of minus spins, one to the left and the other to the right of $C$,
connecting the South side of $A$ to its North side.  By monotonicity
\begin{equation*}
  \pi_A^{\t,+}(\s_x=+\tc \G)-\pi^\t(\s_x=+) \le 0 
\end{equation*}
so that
\begin{equation}
  \label{FM9}
  \pi_A^{\t,+}(\s_x=+)-\pi^\t(\s_x=+)\le \pi_A^{\t,+}(\G^c).
\end{equation}
Let now $\bar A=\{(i,j);\ 1\le i\le L,\; 1\le j\le
2\ceil{(2L+1)^{1/2+\gep}}\}$ so that $\bar A$ consists of just two
copies of $A$ stacked one on top of the other. Then, using
monotonicity together with the standard exponential decay of
correlations in the minus phase $\pi^-_\infty$ (see e.g.
\eqref{eq:standbo}) we get
\begin{equation}
  \label{FM10}
  \bE^\D\(\pi_A^{\t,+}(\G^c)\)\le \nep{-c L^{1/2+\gep}} + \pi^{(-,+,\D)}_{\bar A}(\G^c)
\end{equation}
where the superscript $(-,+,\D)$ indicates the b.c. which is $-$ on the union of
the North boundary and $\D$, and $+$ on the rest of $\partial \bar
A$.  The key equilibrium bound we need at this stage is the
following:
\begin{claim}
\label{second claim}
There exists $c>0$ such that $\pi^{(-,+,\D)}_{\bar A}(\G^c)\le \nep{-c
  L^{3\gep}}$.
\end{claim}
Putting together the bounds we got on the various terms in \eqref{FM7}, we have proved
$\bE^\Delta\|\mu^+_{2t'}-\pi^\tau\|\le c \delta'$ as wished.

The proof of the claim is deferred to the appendix but intuitively the
argument goes as follows. Under the boundary conditions $(-,+,\D)$,
for any configuration $\s\in \O_{\bar A}$ there exist exactly two open
Peierls contours $\g_1,\g_2$ with two possible
scenarios:
\begin{enumerate}[(a)]
\item $\g_1$ joins the two upper corners of $\bar A$ and $\g_2$ the two ends of the interval $\D$;
\item $\g_1$ joins the left upper corner of $\bar A$ with the left
  boundary of $\D$ 
and similarly for
  $\g_2$.
\end{enumerate}
If we recall the definition of the surface tension
\eqref{surfacetension}, the ratio between the probabilities of the two
cases is roughly of the form:
\begin{equation*}
  \nep{-\b\t_\b(\vec e_1) (L +2L^{3\gep}) + 2\b\t_\b(\theta)D} 
\end{equation*}
where $D$ is the Euclidean distance between the left upper corner of $\bar A$
and the left boundary of $\D$ and $\theta$ is the angle formed by the
straight line going through these two points with the horizontal axis.
Clearly $\theta \approx O(L^{-\frac 12 +\gep})$ and $D\approx
L/2-L^{3\gep} + O(L^{2\gep})$. Therefore case (b) is much more likely
than case (a).
\begin{remark}
Notice that it is exactly the presence of the positive correction $O(L^{2\gep})$ in $D$ that forced us to take the length of $\D$ to be $L^{3\gep}$. 
\end{remark}
Once we are in scenario (b) the most likely situation is that neither
$\gamma_1$ nor $\gamma_2$ touch $C$ (otherwise they would have an
excess length of order $L^{3\gep}$) and the desired bound follows by
standard properties of the Ising model with homogeneous boundary
conditions.

{\bf ii)} The proof of $\bE^\Delta\|\mu^-_{2t'}-\pi^\tau\|\le c \delta'$ is
identical, modulo the obvious changes, provided that we redefine the
``massage'' of $\mu_{2t'}^- $ as the censoring in $A,B$ plus the
resetting at time $t'$ of the spins inside $B$ to the value $-1$. A minor observation is that in this case, for the smallness of the
term $ \bE^\D\(\|\nu_1-\pi_A^{\t,-}\|_{C}\)$, we do not need anymore
the global spin flip that was necessary for the dynamics started from
all pluses.  \qed
\begin{remark}
\label{scale}
As we said at the beginning, in order to keep the focus on the main
ideas of the method, Theorem \ref{th:rec} has been given in the
restricted setting in which the length scales are of the form
$L_n=2^n-1$. However it should be clear by now that the case of
arbitrary length scales can be dealt with in a very similar way. A
possible solution requires a slight modification of the definition of
the two inductive statements $\cA(L,t,\d), \cB(L,t,\d)$.

Let $\cF_L$ (respectively $\cG_L$) be the class of rectangles which, modulo
translations, have horizontal base $L$ and height $H\in
[L^{\frac 12+\gep}, (2L)^{\frac 12+\gep}]$ (resp. horizontal base $L$ and
height $H\in [(2L)^{\frac 12+\gep}, (4L)^{\frac 12+\gep}]$). Notice
that any rectangle in $\cG_L$ can be written as the union of two
overlapping rectangles in $\cF_L$ such that the width of their
intersection is still $O(L^{1/2+\gep})$ (as in Figure \ref{fig:Q}).
Moreover for any $n$ large enough and any $L\in [L_{n+1},L_{n+2})$
there exists $L'\in [L_n,L_{n+1})$ such that any rectangle $\L$ in
$\cF_L$ can be written as the union of three sets $A,B,C$ (as in
Figure \ref{fig:R}) where $A\in \cG_{L'}$, $B$ consists of two
disjoint rectangles in $\cG_{L'}$ and $C\equiv\L\setminus B$ satisfies
$dist(C,A^c)=O(L)$ and has horizontal width $O(1)$.

We then say that $\cA'(L,t,\d)$ ($\cB'(L,t,\d)$) holds if \eqref{eq:A(L)} is valid for every rectangle in $\cF_L$ (in $\cG_L$).  It is almost immediate to check that part {\bf (1)} of Theorem \ref{th:rec} continues to hold with this new definition. Part {\bf (2)} can be modified as follows. If $\cB'(L',t,\d)$ holds for every $L'\in [L_n,L_{n+1})$ then $\cA'(L,t_2,\d_2)$ holds for every $L\in [L_{n+1},L_{n+2})$ with $t_2=\nep{c2^{3n\gep}}t$ and $\d_2=c(\d+\nep{-c'2^{3n\gep}})$. The proof of the new version is essentially the same as that given above.  \end{remark}

\section{Proof of the main results}
In what follows we will prove Theorem \ref{th:quadrato} and Corollaries \ref{th:cor++} and \ref{main2}. Notice that, for any $\Lambda\sset \bbZ^2$, any boundary conditions $\t$ and any starting configuration $\s$, $\|\mu_t^\s -\pi^\t\|$ is invariant under the global spin flip $\t\mapsto -\t$ and $\s\mapsto -\s$. Therefore it will be enough to  prove only ``half of the statements''.

\subsection{Proof of Theorem \ref{th:quadrato}}
Recall that $$t_L:=\exp(c L^\gep)$$ for some chosen $\gep>0$ small,
and let $\gep':=\gep/4$. We assume throughout this section that
$L\in\{2^n-1\}_{n\in\N}$.

\subsubsection{Mixing time with ``approximately $(-,-,+,-)$'' boundary conditions}
\label{sec:---+}
First we prove \eqref{eq:maintmix00}-\eqref{eq:maintmix0} when the
b.c. $\tau$ is sampled from a law $\bP$ which is dominated by
$\pi_\infty^-$ on the union of three sides of $\L_L$ and dominates
$\pi^+_\infty$ on the remaining side (e.g. the South border). 

One sees from
\eqref{eq:supgamma}, the definition \eqref{eq:4} of mixing time and
the Markov inequality that \eqref{eq:maintmix00} implies
\eqref{eq:maintmix0}, so we are left with the task of proving 
\eqref{eq:maintmix00}.
This is an almost straightforward
generalization of the proof of point {\bf (1)} of Theorem \ref{th:rec} and
therefore some steps will be  only sketched.

\medskip
For definiteness, we assume that the $L\times L$
square $\L_L$ we are considering is $\{(x_1,x_2)\in \Z^2:1\le
x_1,x_2\le L\}$.
Consider first the evolution started from the $+$
configuration.  For $i\ge0$ let
\begin{eqnarray}
\label{eq:h}
  h_i:=\left\lceil L^{1/2+\gep'}\right\rceil+i\left(\left\lceil(2L+1)^{1/2+\gep'}\right\rceil
    -\left\lceil L^{1/2+\gep'}\right\rceil
  \right).
\end{eqnarray}
To avoid inessential complications, assume that there exists
$k\in\N$ such that $h_{k-1}=L$. Of course,
\begin{eqnarray}
  \label{eq:k}
k\sim \frac{L^{1/2-\gep'}}{2^{1/2+\gep'}-1}.  
\end{eqnarray}
Let $\L_L^i$ be the rectangle of height $h_i$ whose base coincides
with that of $\L_L$, so that in particular $\L_L^{k-1}=\L_L$. We will
prove by induction at the end of the present section that
\begin{Lemma}
\label{th:lemmaindutt}
The following holds for   $i=0,\ldots,k-1$.
  Let the b.c. $\tau$ around the rectangle $\L_L^i$ be sampled from a law
$\bP$ which dominates $\pi^+_\infty$ 
  on the South border and is dominated by $\pi^-_\infty$  on the union of West, East and North borders. Then,
  \begin{eqnarray}
    \bE\|\mu^{+,i}_{(i+1)t_L/k}-\pi^\tau_{\L_L^i}\|
\le (1+i)e^{-c L^{(\gep')^2}}=(1+i)e^{-c L^{\gep^2/16}},
  \end{eqnarray}
where $\mu^{+,i}_t$ is the evolution in $\L_L^i$ started from $+$, 
$\pi^\tau_{\L_L^i}$ is its invariant measure and $c$ depends only on $\beta$ and 
$\gep$.
\end{Lemma}
If the Lemma holds, it is sufficient to apply it for $i=k-1$
to see that $\bE \|\mu_{t_L}^+-\pi^\tau\|\le \exp(-c L^{\gep^2/16})$ as wished.

\medskip

It remains to show that 
\begin{eqnarray}
\label{eq:-ll}
 \bE \|\mu_{t_L}^--\pi^\tau\|\le e^{-c L^{\gep^2/16}}.
\end{eqnarray}
By Theorem \ref{th:PWgeneralizz} and (the analog of)
Lemma \ref{massage}, $\|\mu^-_{t_L}-\pi\|\le
\|\tilde\mu^-_{t_L}-\pi\|$, where this time $\tilde\mu^-_t $ is the
dynamics in $\L_L$ obtained via the following ``massage'': in the time interval
$[0,t_L/2)$ we keep updates only in $B:=R_L^{\gep'} =\{(x_1,x_2)\in
\L_L:x_2\le \lceil L^{1/2+\gep'}\rceil\}$, at time $t_L/2$ we
set to $-$ all spins in $A:=
\{(x_1,x_2)\in\L_L:x_2>\lceil(1/2)L^{1/2+\gep}\rceil\}$ and in
$(t_L/2,t_L]$ we keep updates only in $A$.  In analogy with 
Definition \ref{def:37}, we introduce the
\begin{definition} We let
  \begin{enumerate}[(a)]


  \item $\pi^{\tau,-}_B:=\pi^\tau(\cdot|\si_{B^c}=-);$

  \item $\pi^{\tau,\eta}_A:=\pi^\tau(\cdot|\si_{A^c}=\eta)$;

  \item $\nu_1$ be the distribution obtained at time $t_L/2$;

  \item $\nu_2^\si$ be the distribution obtained at time $t_L$, starting 
at time $t_L/2$ from $\si$ in $A^c$ and from $-$ in $A$.
  \end{enumerate}
\end{definition}

Then, in analogy with \eqref{FM5} one finds
\begin{eqnarray}
\label{eq:triterm}
\bE \|\tilde\mu^-_{t_L}-\pi^\tau\|\le \bE\|\nu_1-\pi^{\tau,-}_B\|_{A^c}+\bE\,\pi^\tau\left(
\|\nu^{\eta_{A^c}}_2-\pi^{\tau,\eta_{A^c}}_A\|
\right)+\bE
\|\pi^{\tau,-}_{B}-\pi^\tau\|_{A^c}.
\end{eqnarray}
From Corollary \ref{th:corolla} one sees
that the first term is smaller than $\exp(-c L^{\gep^2/16})$ (note that
$t_L/2\gg \exp(c L^{3\gep'})$).
The last term in \eqref{eq:triterm} can be
bounded by $\exp(-c' L^{2\gep'})$ (the proof is essentially identical to
the proof of the upper bound on the last term in \eqref{FM5}).
Finally, proceeding like for the second term in \eqref{FM5}, one sees
that
 \begin{eqnarray}
  \bE\, \pi^\tau\left(
\|\nu_2^{\eta_{A^c}}-\pi^{\tau,\eta_{A^c}}_A\|
\right)\le L^2e^{-c' \log (t_L/2)}+e^{-c' L^{2\gep'}}\ll e^{-c' L^{\gep^2/16}}.
 \end{eqnarray}
 Altogether, we proved \eqref{eq:-ll} and the proof of
 \eqref{eq:maintmix} is complete.  \qed

\bigskip

{\sl Proof of Lemma \ref{th:lemmaindutt}.}  
Let for simplicity of notation $\pi^\tau:=\pi^\tau_{\L_L^i}$.
For $i=0$ the claim is
just Corollary \ref{th:corolla} (note that  $\L^0_L=R^{\gep'}_L$).  Assume that the
claim holds for $i-1$.  We define the following three disjoint
rectangles (see Figure \ref{fig:Lemma41}): 
\begin{figure}[htp]
\begin{center}
\includegraphics[width=0.7\textwidth]{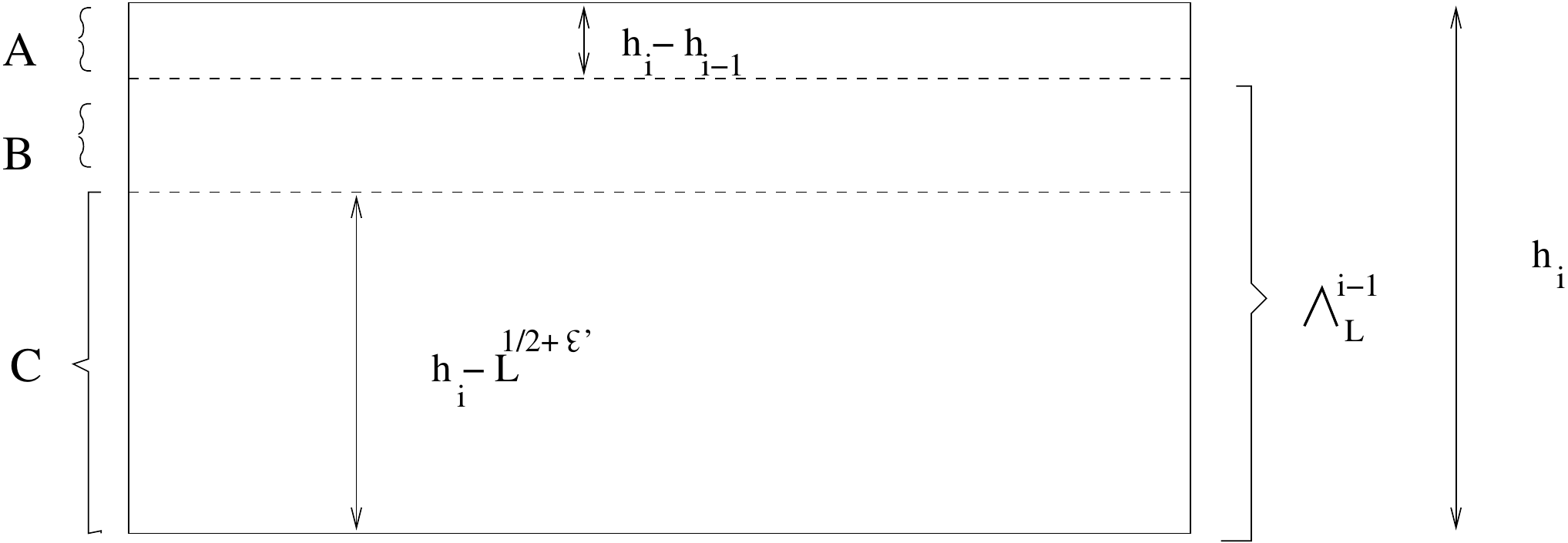}
\end{center}
\caption{The rectangle $\L_L^i$ and its decomposition into $A,B,C$.}
\label{fig:Lemma41} 
\end{figure}
$A:=\L_L^{i}\setminus\L_L^{i-1}$, $C$ is the rectangle
whose South border coincides with that of $\L_L$ and whose height is
$(h_{i}-\lceil L^{1/2+\gep'}\rceil)$, and $B:=\L_L^{i}\setminus(A\cup
C)$.  By Theorem \ref{th:PWgeneralizz} and (the analog of) 
Lemma \ref{massage}, one has
$\|\mu^+_{(i+1)t_L/k}-\pi^\tau\|\le
\|\tilde\mu^+_{(i+1)t_L/k}-\pi^\tau\|$ where the ``massage'' in
$\tilde\mu^+_t$ consists in keeping only the updates in $A\cup B$
in the time interval $[0,t_L/k)$ 
and in $B\cup C$ in the time
interval $(t_L/k,(i+1)t_L/k]$, and setting to $+$ all spins in $B$ at time
$t_L/k$.
In analogy with Definition \ref{def:35}:
\begin{definition}
We let  
\begin{enumerate}[(a)]

\item $\nu_1$ be the distribution obtained at time $t_L/k$, which  assigns zero probability to configurations
which are not all $+$ in $C \cup B$;

\item $\nu^{\si}_2$ be the distribution at time $(i+1)t_L/k$, starting
at time $t_L/k$ from $\si$ in $A$ and from $+$ in $B\cup C$; 

\item $\pi^{\tau,+}_{A\cup B}:=\pi^\tau(\cdot|\si_{C}=+)$;

\item $\pi^{\tau,\eta}_{B\cup C}:=\pi^\tau(\cdot|\eta_{A}=\eta)$;

\item $\pi^{\tau,-}$  be the Gibbs measure in $\L_L^i$ with $-$ b.c. on its
South border and $\tau$ on the other borders.
\end{enumerate}
\end{definition}
One has then
\begin{eqnarray}
 \tilde  \mu^+_{(i+1)t_L/k}(\eta)=\nu_1(\eta_{A})\nu^{\eta_{A}}_2(\eta)
\end{eqnarray}
and
\begin{eqnarray}
\label{eq:mdt}
  \pi^\tau(\eta)=\pi^\tau
(\eta_{A})\pi^{\tau,\eta_{A}}_{B\cup C}(\eta).
\end{eqnarray}
In analogy with \eqref{fm1}
\begin{gather}
  \nonumber
  \|\tilde\mu^+_{(i+1)t_L/k}-\pi^\tau\|
  \le\|\nu_1-\pi^{\tau,+}_{A\cup B}\|_A+\|\gamma-\pi^\tau\|,
\end{gather}
where 
\begin{equation*}
  \g(\h):=\pi^{\tau,+}_{A\cup B}(\h_{A})\nu^{\eta_{A}}_2(\eta).
\end{equation*}
As a consequence, using \eqref{eq:mdt},
\begin{gather}
  \|\mu^+_{(i+1)t_L/k}-\pi^\tau\|\le
\|\nu_1-\pi^{\tau,+}_{A\cup B}\|_A+
\pi^{\tau,-}\left(\|\nu^{\eta_{A}}_2-\pi^{\tau,\eta_{A}}_{B\cup C}
\|
\right)\\\nonumber+
\|\pi^{\tau,+}_{A\cup B}-\pi^\tau
\|_{A}+
\|\pi^\tau-\pi^{\tau,-}\|_{A}.
\end{gather}
Now we can take the expectation with respect to $\bP$.
First of all, we have
\begin{eqnarray}
  \bE \|\nu_1-\pi^{\tau,+}_{A\cup B}\|_A \le e^{-c L^{\gep^2/16}}
\end{eqnarray}
thanks to Corollary \ref{th:corolla}, because $A\cup B$ is a
translation of the rectangle $R^{\gep'}_L$ which appears in the
definition of the claim $\cA(L,t,\delta)$.  As for the
$\bP$-expectation of the third and fourth terms, it is upper bounded
by $\exp(-c L^{2\gep'})$ (the proof is essentially
identical to that of the upper bound for the third and fourth term in
\eqref{FM0}). 
Altogether, the
average of the sum of the first, third and fourth terms is upper
bounded by $\exp(-c L^{\gep^2/16})$.  Finally, in order to bound the
$\bP$-expectation of the second term we need the inductive hypothesis.
Indeed, we can say that
\begin{eqnarray}
\bE \, \pi^{\tau,-}\left(\|\nu^{\eta_{A}}_2-\pi^{\tau,\eta_{A}}_{B\cup C}
\|
\right)\le i e^{-cL^{\gep^2/16}}
\end{eqnarray}
(which concludes the induction step) if we prove that the marginal on
the union of North, East and West borders of $B\cup C$ of the measure
$\bE^-:=\bE\,\pi^{\tau,-}(\cdot)$ is stochastically dominated by
$\pi^-_\infty$.  Indeed, if $(\tau_1,\tau_2,\tau_4) $ is a generic
spin configuration of the North, East and West borders of $B\cup C$ and $f$
is a decreasing function, using monotonicity a couple of times one gets
\begin{eqnarray}
  \bE^-(f)=\bE\,\pi^{\tau,-}(f)\ge\pi^-_\infty(\pi^{\tau,-}(f))\ge
\pi^-_\infty(f),
\end{eqnarray}
which proves the desired stochastic domination.
\qed

\subsubsection{Mixing time with boundary conditions dominated by $\pi_\infty^-$}
Here we prove \eqref{eq:maintmix00} (and therefore, 
 via Markov inequality and \eqref{eq:supgamma},
we obtain \eqref{eq:maintmix0}), when the law $\bP$ 
of $\tau$ is dominated by $\pi^-_\infty$ (or, by spin-flip symmetry, when
 it dominates $\pi^+_\infty$).

\smallskip

We begin with the evolution starting from the $+$ configuration
and we recall that $\L_L=\{1,\ldots,L\}^2$.
One has by monotonicity $\pi^\tau\preceq \mu^+_t$, and therefore
\begin{eqnarray}
  \bE\|\mu^+_{t_L}-\pi^\tau\|&\le& \cS_1+\cS_2:=\sum_{x\in\L_L^-}
\bE\left(\mu^+_{t_L}(\si_x=+)-\pi^\tau(\si_x=+)
\right)\\\nonumber
&&+\sum_{x \in \L_L^+}
\bE\left(\mu^+_{t_L}(\si_x=+)-\pi^\tau(\si_x=+)
\right),
\end{eqnarray}
where $\L_L^-:=\{(i,j)\in\L_L:j<L/2\}$ and $\L_L^+:=\L_L\setminus \L_L^-$.
We will show that the sum $\cS_1$ is small, and $\cS_2$
 can be dealt with similarly.  

Recall that $\L_L^i$ and
$k$ were defined in Section \ref{sec:---+}, and observe that
$\L_L^{\lfloor (3/4) k\rfloor}$ is a rectangle whose base coincides with
that of $\L_L$, and whose height is $h\sim (3/4)L$ (cf.
\eqref{eq:h}-\eqref{eq:k}). Then, thanks to Theorem
\ref{th:PWgeneralizz} (or actually by monotonicity), we know that $\mu^+_{t_L}\preceq \tilde\mu^+_{t_L}$,
where $\tilde\mu_t^+$ is the censored dynamics in which only updates in
$\L_L^{\lfloor (3/4) k\rfloor}$ are retained.  
One has therefore 
\begin{eqnarray}
\label{eq:cS1}
  \cS_1&\le& \sum_{x\in \L_L^-}
\bE\left(\tilde\mu^+_{t_L}(\si_x=+)-\pi^\tau(\si_x=+)\right)\\\nonumber
&\le &
L^2 \left(
\bE \|\tilde\mu^+_{t_L}-\pi^{\tau,+}\|_{\L_L^-}+
\bE \|\pi^\tau-\pi^{\tau,+}\|_{\L_L^-}
\right),
\end{eqnarray}
where $\pi^{\tau,+}$ is the invariant measure of $\tilde\mu^+_t$, \ie
$$\pi^{\tau,+}:=\pi^\tau\left(\left.\cdot\right|\si_{\L_L\setminus 
    \L_L^{\lfloor (3/4) k\rfloor}}=+\right).$$ Since the North border
of $\L_L^-$ is at distance approximately $L/4$ from the North border of
$\L_L^{\lfloor (3/4) k\rfloor}$, the last term in \eqref{eq:cS1} is
easily seen to be upper bounded by $\exp(-c' L)$ (the proof of this
fact is essentially identical to the proof of the upper bound for the
last two terms in \eqref{FM0}).  As for the first term, Lemma
\ref{th:lemmaindutt} (applied with $i=\lfloor (3/4)k\rfloor$) shows
that it is upper bounded by $\exp(-c' L^{\gep^2/16}$). This is because
the evolution $\tilde\mu_t^+$ sees b.c.  $+$ on the North border of
$\L_L^{\lfloor (3/4) k\rfloor}$, and $\tau$ (sampled from $\bP$ which
is stochastically dominated by $\pi^-_\infty$) on the remaining three
borders.  Altogether, we have shown that
\begin{equation*}
  \bE\|\mu^+_{t_L}-\pi^\tau\|\le e^{-c' L^{\gep^2/16}}.
\end{equation*}

\smallskip

Next, we look at the evolution started from all $-$.  Given a site
$x\in \Lambda_L$ and $\ell\in \N$, let $K_\ell$ be the intersection of
$\L_L$ with a square of side $2\ell+1$ centered at $x$.  We let
$\mu^{\tau,-}_{K_\ell,t}$ be the dynamics in $K_\ell$ with $-$ initial
condition and with b.c. $-$ except on $\partial K_\ell\cap\partial
\L_L$, where the b.c. is $\tau$. The invariant measure of such
dynamics is denoted by $\pi^{\tau,-}_{K_\ell}$.  Since
$\mu^-_t\preceq \pi^\tau$, we have
\begin{eqnarray}
 \bE \|\mu_t^--\pi^\tau\|&\le& \sum_{x\in \Lambda_L}
\left[\bE
\mu_t^-(\si_x=-)-\bE\pi^\tau(\si_x=-)
\right]\\
\label{eq:tmixxi}&\le& \sum_{x\in \Lambda_L}
\left[\bE\left(\mu^{\tau,-}_{K_\ell,t}(\si_x=-)-\pi^{\tau,-}_{K_\ell}(\si_x=-)
\right)
  +e^{-c \ell}\right]\\
\label{eq:daboot}
&\le&\sum_{x\in \Lambda_L}
\left(\bE\,\|\mu^{\tau,-}_{K_\ell,t}-\pi^{\tau,-}_{K_\ell}\|
%
+e^{-c \ell}
\right).
\end{eqnarray}
The ``error term'' $\exp(-c\ell)$ comes from comparing $\bE\pi^\tau(\si_x=-)$ and
$\bE \pi^{\tau,-}_{K_\ell}(\si_x=-)$ (see the proof of Claim \ref{claim2} for very similar arguments).
We know from \cite[Corollary 2.1]{cf:M_2D} that $T_{{\rm
    mix},K_\ell}^{\tau,-}\le e^{c\ell}$, uniformly in $\tau$.
Therefore, from 
\eqref{eq:*} and choosing $t=t_L$ and
$\ell=\lceil\frac1c(\log t-\log \log t)\rceil\approx L^\gep$, one gets 
\begin{eqnarray}
\label{eq:mu+t}
 \bE \|\mu_{t_L}^--\pi^\tau\|\le e^{-cL^\gep}.
\end{eqnarray}
\qed

\subsection{Proof of Corollary \ref{th:cor++}}
We restart from \eqref{eq:daboot}, which in the case of $\tau\equiv -$ gives
\begin{eqnarray}
  \label{eq:10}
  \|\mu_t^--\pi^-\|\le |\L_L|e^{-c\ell}+\sum_{x\in\L_L} \|\mu^{-}_{K_\ell,t}-\pi^{-}_{K_\ell}\|
\end{eqnarray}
where $\pi^-,\mu^{-}_{K_\ell,t}$ and $\pi^{-}_{K_\ell}$ are just $\pi^{\tau}, \mu^{\tau,-}_{K_\ell,t}$ and $\pi^{\tau,-}_{K_\ell}$ respectively, in the 
specific case $\tau\equiv -$.
Now we use the extra information that
the mixing time $T^-_{{\rm mix},K_\ell}$ of the dynamics $\mu^{-}_{K_\ell,t}$ is at most $\exp(c'\ell^\gep)$, as follows from 
\eqref{eq:maintmix}.
We choose $\ell$ to be the smallest integer in
the sequence $\{2^n-1\}_{n\in\N}$ such that $c\ell>3\log L$, so that the first term in the r.h.s. of \eqref{eq:10} is smaller than $1/L$.
Taking 
$t_1:=\exp(c(\log L)^\gep)$, one has from \eqref{eq:*}
\begin{eqnarray}
  \label{eq:11}
 \|\mu^{-}_{K_\ell,t_1}-\pi^{-}_{K_\ell}\|\le e^{-t_1/T^-_{{\rm mix},K_\ell}}\le \exp[- \exp(c(\log L)^\gep-c'\ell^\gep)]\ll 1/|\L_L|
\end{eqnarray}
if one chooses $c$ suitably larger than $c'$ (recall that we chose $\ell=O(\log L)$)
and the corollary is proved.
\qed

\subsection{Proof of Corollary \ref{main2}}

This is rather standard, once \eqref{eq:maintmix} is known (cf. for
instance Theorem 3.2 in \cite{cf:M_2D} or Theorem 3.6 in
\cite{cf:CMT}). Clearly, it is sufficient to prove the result with $f$
redefined as  $f(\si):=(\si_0+1)$ which has the advantage of being
non-negative, increasing and with support $\{0\}$.
Consider a square $J_\ell\subset \Z^2$ with side 
$2\ell+1\in \{2^n-1\}_{n\in\N}$ and centered at $0$. 
By the exponential decay of correlations in the pure
phase $\pi^+_\infty$,
\begin{eqnarray}
\label{eq:mumu}
|  \pi^+_\infty(f)-\pi^+_{J_\ell}(f)|\le c \,e^{-c'\ell}.
\end{eqnarray}
Moreover, by monotonicity, for every initial configuration $\si$
of the infinite system
\begin{eqnarray}
  0\le (e^{t\mathcal L} f)(\si)\le \left(e^{t \mathcal L^{+}_{J_\ell}} 
f\right)(\si)
\end{eqnarray}
and the right-hand side is an increasing function of $\si$; in accord
with the notations of Section \ref{sect:Glauber}, $\mathcal
L^{+}_{J_{\ell}}$ denotes the generator of the dynamics in $J_\ell$
with $+$ boundary conditions on $\partial J_\ell$ (its invariant
measure is of course $\pi^+_{J_\ell}$) and $\cL$ is the generator of
the infinite-volume dynamics.  One has then (using once more
monotonicity)
\begin{eqnarray}
  \pi^+_\infty\left[(e^{t\mathcal L} f)^2\right]
\le \pi^+_{J_\ell}\left[\left(e^{\mathcal L^{+}_{J_\ell}} f\right)^2\right]
\end{eqnarray}
which, together with \eqref{eq:mumu}, gives
\begin{eqnarray}
\rho(t)=  \Var^+_\infty\left(e^{t\mathcal L} f\right)\le 
\Var_{\pi^+_{J_\ell}}\left(e^{t\mathcal L^{+}_{J_\ell}} f\right)+
c\, e^{-c' \ell}.
\end{eqnarray}
By \eqref{eq:5}, one has that 
\begin{eqnarray}
  \Var_{\pi^+_{J_\ell}}\left(e^{t\mathcal L^{+}_{J_\ell}} f\right)\le 
\Var_{\pi^+_{J_\ell}}(f)
e^{-2t\gap^+_{J_\ell}},
\end{eqnarray}
with $\gap^+_{J_\ell}$ the spectral gap of  $\mathcal
L^{+}_{J_\ell}$.  From the  inequality
\begin{eqnarray}
  \gap\ge \frac1{\tmix}
\end{eqnarray}
(cf. \eqref{eq:6})
and \eqref{eq:maintmix}, one deduces that for every $\gep>0$
\begin{eqnarray}
\label{eq:onesis}
   \Var^+_\infty(e^{t\mathcal L} f)\le c\,\left(
e^{-c'\ell}+e^{-2t e^{-c\ell^{\gep}}}
\right).
\end{eqnarray}
Now letting $\ell=\ell(t)$ be the smallest integer  such that
\begin{eqnarray}
c \ell^{\gep}\ge\log t-\frac1\gep \log\log t,
\end{eqnarray}
(with the condition that
$2\ell+1\in\{2^n-1\}_{n\in\N}$)
one sees  that \eqref{eq:onesis} implies 
\eqref{eq:decorrel}.
\qed

\appendix

\section{Some equilibrium estimates}
\label{sec:equilibrio}

\subsection{A few basic facts on cluster expansion}
\label{sec:stand}
In this section we rely  on the results of \cite{cf:DKS},
but we try to be reasonably self-contained.
We let ${\Z^2}^*$ be the dual lattice of $\Z^2$ and we call a {\sl bond}
any segment joining two neighboring sites in ${\Z^2}^*$.
Two sites $x,y$ in $\Z^2$ are said to be {\sl separated by a bond $e$} if
their distance (in $\bbR^2$) from $e$ is $1/2$. A pair of orthogonal
bonds which meet in a site $x^*\in {\Z^2}^*$ is said to be a 
{\sl linked pair of bonds} if both bonds are on the same side of the 
forty-five degrees line across $x^*$. A {\sl contour} is a
sequence $e_0,\ldots,e_n$ of bonds such that:
\begin{enumerate}
\item $e_i\ne e_j$ for every $i\ne j$, except possibly when
$(i,j)=(0,n)$

\item for every $i$, $e_i$ and $e_{i+1}$ have a common vertex in ${\Z^2}^*$

\item if four bonds $e_i,e_{i+1}$ and $e_j,e_{j+1},i\ne j,j+1$ intersect at some $x^*\in {\Z^2}^*$,
then $e_i,e_{i+1}$ and $e_j,e_{j+1}$ are linked pairs of bonds.

\end{enumerate}
If $e_0=e_n$, the contour is said to be {\sl closed}, otherwise it is
said to be {\sl open}.  Given a contour $\gamma$, we let $\Delta
\gamma$ be the set of sites in $\Z^2$ such that either their distance
(in $\bbR^2$) from $\gamma$ is $1/2$, or their distance from the set
of vertices in ${\Z^2}^*$ where two non-linked bonds of $\gamma$ meet
equals $1/\sqrt2$.

We need the following
\begin{definition}
\label{def:VVbar}
  Given $V\subset \Z^2$, we let $\tilde V\subset \R^2$ be 
the union of all closed unit squares
centered at each site in $V$, and $\bar V$ be the set of all bonds
$e\in {\Z^2}^*$ such that at least one of the two sites separated by $e$ belongs
to $V$.
\end{definition}

Given a rectangular domain $V\subset \Z^2$, a configuration $\si\in
\Omega_V$ and a boundary condition $\tau$ on $\partial V$, let
$\si^{(\tau,+)}$ be the spin configuration on $\Z^2$ which coincides
with $\si$ in $V$, with $\tau$ on $\partial V$ and which is $+$
otherwise. One immediately sees that the (finite) collection of bonds
of ${\Z^2}^*$ which separate neighboring sites $x,y\in \Z^2$ such that
$\si^{(\tau,+)}_x\ne\si^{(\tau,+)}_y$ splits in a unique way into a
finite collection $\Gamma^\tau(\si)$ of closed contours. It is easy to
see that $\Gamma^\tau(\si)\cap \tilde V$ consists of a certain number
of closed contours, plus $m$ open contours, where $m$ is such that
going along $\partial V$ one meets $2m$ changes of sign in $\tau$.
Note that the collection of the $2m$ endpoints of the open contours is
fixed uniquely by $\tau$. We write $\Gamma_{open}^\tau(\si)$ for the
collection $\{\gamma_1,\ldots,\gamma_m\}$ of open contours in
$\Gamma^\tau(\si)\cap \tilde V$.  Of course, the open contours
$\gamma_i$ have to satisfy certain compatibility conditions:
$\gamma_i$ and $\gamma_j$ have no bond in common if $i\ne j$, and if
they meet at some $x^*\in {\Z^2}^*$, each of the two linked pairs of
bonds belongs to only one contour. Moreover, each $\gamma_i$ is
contained in $\tilde V$ and the collection of the endpoints of the
$\{\gamma_i\}_{i\le m}$ must coincide with that dictated by $\tau$. We
will write $\{\gamma_1,\ldots,\gamma_m\} \sim\tau$ to indicate that
the collection of open contours is compatible with $\tau$.

The following result can be easily deduced from \cite[Sec. 3.9 and
4.3]{cf:DKS}.  Writing as usual $\pi^{\tau}_V$ for the
equilibrium measure in $V$ with b.c. $\tau$, one has
\begin{Theorem}
\label{th:clustexp}
  There exists $\beta_0$ such that for every $\beta>\beta_0$ the
  following holds. For every rectangle $V\subset \Z^2$, every  b.c.
  $\tau$ on $\partial V$ and every collection
  $\{\gamma_1,\ldots,\gamma_m\}$ of open contours compatible with
  $\tau$, one has
\begin{eqnarray}
  \label{eq:clustexp} \pi^{\tau}_V
  \left(\si:\Gamma_{open}^\tau(\si)= \{\gamma_1,\ldots,\gamma_m\}\right)=
\frac {\Psi(\{\gamma_1,\ldots,\gamma_m\};V)}{\Xi(V,\tau)}
\end{eqnarray}
where the Boltzmann weight $\Psi(\{\gamma_1,\ldots,\gamma_m\};V)$ is defined
as
\begin{eqnarray}
{ \Psi}(\{\gamma_1,\ldots,\gamma_m\};V):=
  \exp
  \left\{
    -2\beta\sum_{i=1}^m|\gamma_i|-
    \sumtwo{\Lambda\subset V:}{\Lambda\cap(\cup_i\Delta\gamma_i)\ne\emptyset}\Phi(\Lambda)
  \right\},
\end{eqnarray}
$|\gamma_i|$ is the geometric length of $\gamma_i$ and
\begin{eqnarray}
  \label{eq:Xi}
  \Xi(V,\tau):=\sum_{\{\gamma_1,\ldots,\gamma_m\}
\sim\tau}\Psi(\{\gamma_1,\ldots,\gamma_m\};V).
\end{eqnarray}
The potential $\Phi$ satisfies for every $\Lambda\subset
V,|\Lambda|\ge2$ and for every $x\in V$:
\begin{eqnarray}
\label{eq:Phi}
&&  \left|\Phi(\Lambda)\right|\le \exp(-2(\beta-\beta_0)d(\Lambda))\\
&&\left|\Phi(\{x\})\right|\le \exp(-8(\beta-\beta_0))
\end{eqnarray}
where, for connected (in the sense of subgraphs of the graph $\Z^2$)
$\Lambda$, $d(\Lambda)$ is the length of the smallest connected set of
bonds from $\bar \Lambda$ (cf. Definition \ref{def:VVbar})
containing all the bonds which separate
sites in $\Lambda$ from sites in $\Lambda^c$.  If $\Lambda$ is not
connected then $d(\Lambda):=+\infty$.
\end{Theorem}

The fast decay property of $\Phi$  (with respect to both
$\beta$ and $d(\L)$) has the following simple consequence:
 \begin{Lemma}\cite[Lemma 3.10]{cf:DKS}
\label{th:lemma3102}
There exists $\beta'_0$ depending only on $\beta_0$ of Theorem
\ref{th:clustexp} such that for $\beta>\beta'_0$, for every bond $e\in
{\Z^2}^*$ and for every $d>0$ one has
 \begin{eqnarray}
   \sumtwo{\Lambda\subset \Z^2:e\in \bar \Lambda}{d(\Lambda)\ge d}
 e^{-2(\beta-\beta_0)d(\Lambda)}\le e^{-2(\beta-\beta'_0)d}.
 \end{eqnarray}
\end{Lemma}
This allows to essentially neglect the
interaction between portions of a contour which are sufficiently far 
from each other.

\medskip

In order to apply directly results from \cite{cf:DKS} to obtain the
estimates we need, we define the {\sl canonical ensemble of contours}.
Let $a,b$ be sites in $ {\Z^2}$. Then, for any open contour $\gamma$
which has $a+(1/2,1/2),b +(1/2,1/2)\in {\Z^2}^*$ as endpoints, in
formulas $a\stackrel \gamma\leftrightarrow b$ (with some abuse of
language, we will sometimes say that $\gamma$ {\sl connects $a$ and
  $b$}), we define the probability distribution
\begin{eqnarray}
\label{eq:canonico}
  \mathcal P_{a,b}(\gamma):=
\left(\mathcal Z_{a,b}\right)^{-1}\exp\left\{
-2\beta|\gamma|-\sumtwo{\Lambda\subset \Z^2:}{\Lambda\cap \Delta\gamma\ne
\emptyset}\Phi(\Lambda)
\right\}=\left(\mathcal Z_{a,b}\right)^{-1}\,\Psi(\gamma; {\Z^2})
\end{eqnarray}
and of course
\begin{eqnarray}
  \label{eq:Zab}
  \mathcal Z_{a,b}:= \sum_{\gamma:a\stackrel \gamma\leftrightarrow b}
\Psi(\gamma;\Z^2).
\end{eqnarray}
Note that we do not require that $\gamma\subset \tilde V$ and the
sum in $\Psi$ is now over all (connected) sets $\Lambda\subset \Z^2$.
The expectation w.r.t. $\mathcal P_{a,b}$ will be denoted by $\mathcal
E_{a,b}$.

\subsubsection{Surface tension and basic properties}
Let ${\vec n}$ be a vector in the unit circle $\mathbb S$ such that
$\vec n \cdot \vec e_1>0$ and call $\phi_{\vec n}$ the angle it forms
with $\vec e_1$ (of course, $-\pi/2<\phi_{\vec n}< \pi/2$). For $N\in
\N$, let $b_{N,{\vec n}}=(N,y_{N,{\vec n}})\in {\Z^2}$ where
$y_{N,{\vec n}}=\max\{y\in \Z:y\le N\tan(\phi_{\vec n})\}$. Let also
$\underline 0:=(0,0)$.  Then, it is known \cite[Prop. 4.12]{cf:DKS}
that, for $\beta$ large enough, the surface tension introduced in
\eqref{surfacetension} is given by
\begin{eqnarray}
\label{eq:surftens}
  \tau_\beta({\vec n}):=-\lim_{N\to\infty}\frac1{\beta d(\underline0,b_{N,
\vec n})}
\log \mathcal Z_{\underline0,b_{N,{\vec n}}},
\end{eqnarray}
where, if $x,y\in\R^2$, $d(x,y)$ is their Euclidean distance.  To be
precise, one has to assume that $\phi_{\vec n}$ is bounded away from
$\pm\pi/2$ uniformly in $N$, but this will be inessential for us since
we will always have $\phi_{\vec n}$ small.

One can extract from \cite[Sec. 4.8, 4.9 and 4.12]{cf:DKS} that the
surface tension is an analytic function of $\phi_{\vec n}$ (always
assuming that $\beta$ is large enough), and by symmetry one sees that
it is an even function of $\phi_{\vec n}$.  In \cite[Sec.
4.12]{cf:DKS}, sharp estimates on the rate of convergence in
\eqref{eq:surftens} (e.g.  \eqref{eq:tausharp} below) are given.

\subsection{Proof of \eqref{eq:primaeq}}
\label{sec:primaeq}
The domain $E_L(Q_L)$ which appears in \eqref{eq:primaeq} is a
rectangle with height shorter than its base, and the b.c. $\tau$ is
$+$ on the South border and $-$ otherwise.  Since the event that the
unique open contour reaches the height of  the South border of $A$ is increasing, in
order to prove \eqref{eq:primaeq}, by the FKG inequalities we can
first of all move upwards the North border of $E_L(Q_L)$ until we
obtain a square (of side $3L$, which however here we call just $L$);
we let therefore $V:=\{1,\ldots,L\}^2$.  Secondly (always by FKG) we
can change the b.c.  $\tau $ to $\tau'\ge \tau$ by first fixing a $\delta>0$
and then establishing that $\tau'_x=+$ if $x=(x_1,x_2)\in\partial V$
with $x_2\le \lfloor\delta L^{1/2+\gep}\rfloor$, and $ \tau'_x=-$
otherwise.

Given a configuration $\si\in\O_V$, let $\gamma$ be the unique open
contour in $\Gamma_{open}^{\tau'}(\si)$: of course, $\gamma\subset
\tilde V$ and $a_1\stackrel \gamma\leftrightarrow a_2$, where
$a_1:=(0,\lfloor \delta L^{1/2+\gep}\rfloor)$ and $a_2:=(L,
\lfloor\delta L^{1/2+\gep}\rfloor)$.  We let $h(\gamma):= \max\{x_2:
(x_1,x_2)\in \gamma\}$ be the maximal height reached by $\gamma$,
while as usual $\gep>0$ is small and fixed.  Looking at
\eqref{eq:clustexp} and \eqref{eq:primaeq}, we see that what we have
to prove is that for every fixed $\delta>0$ one has for every $L\in
\N$
\begin{eqnarray}
\label{eq:ND}
\frac{\cN}{\Xi(V,\tau')}:=  \frac{
\sum_{\gamma\sim\tau'}
\Psi(\gamma;V){\bf 1}_{\{h(\gamma)>2\delta L^{1/2+\gep}\}}}
{\Xi(V,\tau')}\le e^{-c L^{2\gep}}
\end{eqnarray}
for some $c(\beta,\delta,\gep)>0$.
We will always assume that $\beta$ is large enough.

First we upper bound the numerator in \eqref{eq:ND}:  with the
notations of Section \ref{sec:stand} (cf. in particular \eqref{eq:canonico}) and 
setting for a given contour $\gamma$ and a given $V\subset \Z^2$
\begin{eqnarray}
  \label{eq:8}
  \Phi_{V}(\gamma):=\sum_{\L\subset \Z^2:\,\Lambda\cap \Delta\gamma\ne
\emptyset,\L\cap V^c\ne\emptyset}\Phi(\L),
\end{eqnarray}
one has
\begin{eqnarray}
  \cN&\le&\cZ_{a_1,a_2}\,\cE_{a_1,a_2}\left[
{\bf 1}_{\{h(\gamma)>2\delta L^{1/2+\gep}\}}\, 
\exp\left(\Phi_{V}(\gamma)
\right)
\right]\\\nonumber
&\le& 
\cZ_{a_1,a_2}\sqrt{
\cP_{a_1,a_2}(h(\gamma)>2\delta L^{1/2+\gep})
}\sqrt{\cE_{a_1,a_2}\left[
\exp\left(2 \Phi_{V}(\gamma)
\right)
\right]
},
\end{eqnarray}
where in the first step we simply removed the constraint that
$\gamma\subset \tilde V$, which is implicit in the requirement
$\gamma\sim \tau'$.  It follows directly from \cite[Prop.
4.15]{cf:DKS} that the first  square root is smaller than
$\exp(-c L^{2\gep})$ (note that we are requiring the contour to reach
a height which exceeds by $\delta L^{1/2+\gep}$ the height of its
endpoints). On the other hand, from \cite[Th. 4.16, in particular Eq.
(4.16.6)]{cf:DKS} and the fast decay properties of $\Phi$ 
(in particular Lemma \ref{th:lemma3102}) it is not
difficult to deduce that the second one is upper bounded by
$
\exp\left(c (\log L)^c\right).
$
  Moreover,
one has \cite[Eq. (4.12.3)]{cf:DKS} that
\begin{eqnarray}
\label{eq:tausharp}
  \cZ_{a_1,a_2}\le c(\beta)\frac{e^{-\beta \tau_\beta(\vec e_1)L}}{\sqrt L},
\end{eqnarray}
where of course $\tau_\beta(\vec e_1)$ is the surface tension in the 
horizontal direction and we used the fact that $d(a_1,a_2)= L$.
In conclusion, we have 
\begin{eqnarray}
\label{eq:ubN}
\cN\le
\exp\left[-\beta\tau_\beta(\vec e_1)L-c L^{2\gep}\right].  
\end{eqnarray}

Next we observe that, again from \cite[Th. 4.16 and Eq.
(4.16.7)]{cf:DKS},
\begin{eqnarray}
\label{eq:lbD}
 \Xi(V,\tau') \ge \exp\left[-\beta\tau_\beta(\vec e_1)L-c (\log L)^c\right]
\end{eqnarray}
 which together with \eqref{eq:ubN} concludes the proof of \eqref{eq:primaeq}.
 \qed

\subsection{Proof of Claim \ref{second claim}}
In this section, $V$ is the rectangle $\{(i,j)\in\Z^2: 1\le i\le L, 1
\le j\le 4\lceil (2L+1)^{1/2+\gep}\rceil\}$ and the b.c. $\tau$ is
defined by $\tau_x=-$ for $x\in \Delta:= \{(i,0)\in\Z^2: |i-\lfloor L/2
\rfloor|\le
L^{3\gep}\}$ and for $x=(x_1,x_2)\in \partial V$ with $x_2> 2\lceil
(2L+1)^{1/2+\gep}\rceil$; $\tau_x=+$ otherwise.  Moreover, $C$ is the
infinite vertical column $C=\{(x_1,x_2)\in\bbR^2:x_1=\lfloor
L/2\rfloor\}$.  Write $\Delta_1+(1,0)$ (resp. $\Delta_2$) for the
left-most (resp.  right-most) point in $\Delta$. For every
$\si\in\O_V$ there are two open contours in $\Gamma_{open}^\tau(\si)$:
$\gamma_1$ and $\gamma_2$, and we establish by convention that
$\gamma_1$ is the contour which contains $\Delta_1+(1/2,1/2)$ as one
of its endpoints.  Two cases can occur (see Figure
\ref{fig:twocases}):
\begin{itemize}
\item either $\Delta_1\stackrel {\gamma_1} \leftrightarrow \Delta_2$ and 
$w_1 \stackrel {\gamma_2} \leftrightarrow w_2$,
where $w_1:=(0, 2\lceil(2L+1)^{1/2+\gep}\rceil)$ and
$w_2:=(L, 2\lceil(2L+1)^{1/2+\gep}\rceil)$,

\item or $w_1 \stackrel {\gamma_1} \leftrightarrow \Delta_1$
and $\Delta_2\stackrel {\gamma_2} \leftrightarrow w_2$.

\end{itemize}
\begin{figure}[htp]
\begin{center}
\includegraphics[width=0.9\textwidth]{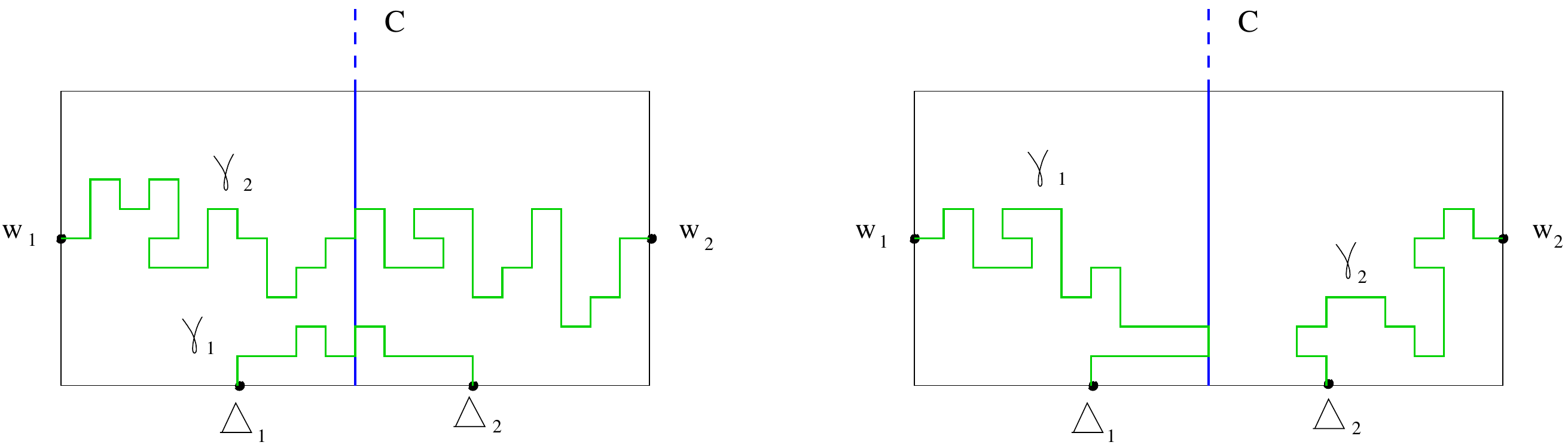}
\end{center}
\caption{The two topologically distinct possibilities: either
  $\gamma_1$ connects $\Delta_1$ to $\Delta_2$, or it connects $w_1$
  to $\Delta_1$. The fist case is very unlikely, see
  \eqref{eq:tecnica1}.  }
\label{fig:twocases} 
\end{figure}

Let $C_1$ (resp. $C_2$) be the vertical column at distance $\lfloor 
L^\gep\rfloor$ to the left (resp. to the right) of the column $C$.
Then, one has the
\begin{Lemma}
\label{th:FKG}
  The probability that appears in Claim \ref{second claim} can be upper 
bounded as 
\begin{eqnarray}
  \pi_{\bar A}^{(-,+,\Delta)}(\Gamma^c)\le \pi_V^{\tau}(\bar \Gamma^c),
\end{eqnarray}
where
\begin{eqnarray}
  \label{eq:barL}
  \bar \Gamma:=\{w_i \stackrel {\gamma_i} \leftrightarrow \Delta_i
\mbox{\; and\;}
 \gamma_i\cap C_i=\emptyset,
i=1,2
\}.
\end{eqnarray}
\end{Lemma}

Therefore, from Theorem \ref{th:clustexp} we see that to prove
Claim \ref{second claim}  it is enough to show that
\begin{eqnarray}
\label{eq:tecnica1}
\frac{\cN_1}{\Xi(V,\tau)}:=
  \frac{\sum_{\{\gamma_1,\gamma_2\}\sim\tau}\Psi(\{\gamma_1,\gamma_2\};V)
{\bf 1}_{\{\Delta_1\stackrel {\gamma_1} \leftrightarrow \Delta_2\}}}
{\Xi(V,\tau)}\le e^{-c L^{3\gep}}
\end{eqnarray}
and that
\begin{eqnarray}
\label{eq:tecnica2}
\frac{\cN_2}{\Xi(V,\tau)}:=
   \frac{\sum_{\{\gamma_1,\gamma_2\}\sim\tau}\Psi(\{\gamma_1,\gamma_2\};V)
{\bf 1}_{\{\Delta_1\stackrel {\gamma_1} \leftrightarrow w_1\}}
{\bf 1}_{\{\gamma_1\cap C_1\ne\emptyset 
\}}
}
{\Xi(V,\tau)}\le e^{-c L^{3\gep}},
\end{eqnarray}
for some positive $c=c(\beta,\gep)$.

{\sl Proof of Lemma \ref{th:FKG}.}
Since the event $\Gamma^c$ is increasing, we note first of all that thanks 
to FKG we can enlarge the
system from $\bar A$ to $V$ and change the b.c. from $(-,+,\Delta)$ to 
$\tau$. Secondly, we observe that the event $\bar \Gamma$ implies 
$\Gamma$.
\qed

\subsubsection{Lower bound on $\Xi(V,\tau)$}
We will prove that there exists a positive constant $c'$ such that for
$\beta$ large
\begin{eqnarray}
\label{eq:LBXi}
  \Xi(V,\tau)\ge \exp\left(-\beta\tau_\beta(\vec e_1)(L-c' L^{3\gep})\right).
\end{eqnarray}
Since we want a lower bound, we are allowed to keep only the
configurations $\{\gamma_1,\gamma_2\}\sim \tau$ such that
$w_i\stackrel{\gamma_i}\leftrightarrow \Delta_i$ and $\gamma_i$ does
not touch the column $C_i$, for $i=1,2$.  Call $\cG_i, i=1,2$ 
the set of configurations of $\gamma_i$  allowed
by the above constraints.

Using the decay properties of
$\Phi$, one sees that
\begin{eqnarray}
\label{eq:denomina}
  \Xi(V,\tau)\ge c\,\left(\sum_{\gamma_1\in\cG_1}\Psi(\gamma_1;V)
\right)^2.
\end{eqnarray}
The square is due to the fact that $\gamma_1$ and $\gamma_2$ essentially 
do not interact because their mutual distance is
larger than $L^\gep$ (the residual interaction can be bounded by a constant
which is absorbed in $c$).
It remains to prove that 
\begin{eqnarray}
  \sum_{\gamma_1\in\cG_1}
\Psi(\gamma_1;V)
\ge \exp(-\beta\tau_\beta(\vec e_1)((L/2)-c' L^{3\gep}))
\end{eqnarray}
for some positive $c'$.  This is an immediate consequence of Lemma
\ref{th:lemmastrazio} below (applied with $\kappa=\gep$), together
with the fact that $d(w_1,\Delta_1)=L/2- L^{3\gep}+O(L^{2\gep})$, of
the fact that the angle $\phi$ formed by the segment $w_1\Delta_1$ and
$\vec e_1$ is $O(L^{-1/2+\gep})$, and finally of the analyticity of
the surface tension and its symmetry around $\vec e_1$.

\subsubsection{Upper bound on $\cN_1$}
Using rough upper bounds on the number of paths $\gamma_1$ which
connect $\Delta_1$ and $\Delta_2$ and the decay properties of $\Phi$
(in particular Lemma \ref{th:lemma3102}), one sees that for $L$ large
\begin{eqnarray}
  \cN_1\le e^{-c L^{3\gep}}\sum_{\gamma\subset \tilde V:\,w_1\stackrel\gamma \leftrightarrow
    w_2}\Psi(\gamma;V)
\end{eqnarray}
for some $c=c(\beta,\gep)>0$, where of course one uses the fact that  $
d(\Delta_1,\Delta_2)=2L^{3\gep}$. Moreover, Theorem  4.16 of
\cite{cf:DKS} ensures that 
\begin{eqnarray}
  \label{eq:ensure}
  \sum_{{\gamma\subset \tilde V}:\,w_1\stackrel\gamma \leftrightarrow
    w_2}\Psi(\gamma;V)\le \exp(-\beta \tau_\beta(\vec e_1) L+c(\log
  L)^c),   
\end{eqnarray}
which, together with \eqref{eq:LBXi}, concludes the proof of
\eqref{eq:tecnica1}.

\subsubsection{Proof of \eqref{eq:tecnica2}}
The estimate we wish to prove is very intuitive: if the path
$\gamma_1$ makes a deviation to the right to touch the column $C_1$,
it has an excess length, and therefore an excess energy, of order
$L^{3\gep}$ with respect to typical paths. The actual proof of
\eqref{eq:tecnica2} is a straightforward (although a bit lengthy)
application of results from \cite{cf:DKS} and of the FKG inequalities. We
sketch only the main steps.

First of all, letting $d(\gamma_1,\gamma_2):=\min\{d(x_1,x_2),x_i\in
\gamma_i, i=1,2\}$, we show that the contribution of the configurations
such that $d(\gamma_1,\gamma_2)< L^\gep$ is negligible.
To this purpose, decompose first of all $\cN_2$ as
$  \cN_2=\cN_2'+\cN_2''$ where 
\begin{eqnarray}
  \cN_2':=\sum_{\{\gamma_1,\gamma_2\}\sim \tau}\Psi(\{\gamma_1,\gamma_2\};V)
{\bf 1}_{\{\Delta_1\stackrel {\gamma_1} \leftrightarrow w_1\}}
{\bf 1}_{\{\gamma_1\cap C_1\ne\emptyset\}}
{\bf 1}_{\{d(\gamma_1,\gamma_2)< L^\gep\}}.
\end{eqnarray}
Consider the paths $\gamma_i$ as oriented from $w_i$ to $\Delta_i$
and, if $d(\gamma_1,\gamma_2)< L^\gep$, call
$P:=P(\gamma_1,\gamma_2):=(x_1,x_2)\in {\Z^2}^*\times {\Z^2}^*$ where
$x_1$ is the first point in $\gamma_1\cap {\Z^2}^*$ which is at
distance less than $L^\gep$ from $\gamma_2$, and $x_2$ is the first
point in $\gamma_2\cap {\Z^2}^*$ at distance less than $L^\gep$ from
$x_1$. Of course, $P$ can take at most $L^2$ different values (this is
a rough upper bound) and we can decompose $\cN_2'$ as $\cN_2'=\sum_p \cN_{2,p}'$
where $\cN_{2,p}'$ contains only the terms such that
$P(\gamma_1,\gamma_2)=p$.  Given $(\gamma_1,\gamma_2)$ such that
$P(\gamma_1,\gamma_2)=p$, for $i=1,2$ one can write $\gamma_i$
as the union of $\gamma_i'$ and $\gamma_i''$, where $\gamma_i'$
connects $w_i$ to $x_i$, and $\gamma_i''$ connects $x_i$ to
$\Delta_i$.
Using the decay properties of $\Phi$ one sees that, uniformly in 
$p$ and in $\{\gamma_i'\}_{i=1,2}$,
\begin{eqnarray}
  \sum_{\{\gamma_i''\}_{i=1,2}}\Psi(\{\gamma_1,\gamma_2\};V)\le c
\Psi(\gamma_1';V)\Psi(\gamma_2';V),
\end{eqnarray}
where the sum runs over all the configurations of
$\{\gamma_i''\}_{i=1,2}$ compatible with $\{\gamma_i'\}_{i=1,2}$.  Let
$\Sigma$ be the set of paths $\gamma_3$ which connect $x_1$ to $x_2$,
and such that the concatenation of $\gamma_1',\gamma_3$ and $\gamma_2'$ is an
admissible open path, call it simply $\gamma$, connecting $w_1$ to $w_2$ 
and
contained in $\tilde V$.
Of course, the set $\Sigma$ depends on $\{\gamma_i'\}_{i=1,2}$.  Then, one sees
 that
\begin{eqnarray}
   \sum_{\{\gamma_i''\}_{i=1,2}}\Psi(\{\gamma_1,\gamma_2\};V)\le
e^{c L^\gep} \sum_{\gamma_3\in \Sigma}\Psi(\gamma;V).
\end{eqnarray}
In conclusion, summing over the admissible configurations of 
$\{\gamma_i'\}_{i=1,2}$ and over the possible values of $p$, recalling
\eqref{eq:ensure}
and the lower bound \eqref{eq:LBXi}, we have shown that
\begin{eqnarray}
\label{eq:A26}
  \frac{\cN_2'}{\Xi(\tau,V)}\le e^{-cL^{3\gep}}.
\end{eqnarray}

As for $\cN_2''$, using the decay properties of
the potential $\Phi$ one sees immediately that, since
$d(\gamma_1,\gamma_2)\ge L^\gep$, the mutual interaction between the
two paths can be bounded by a constant, so that
\begin{eqnarray}
  {\cN_2''}\le c\;
  \sum_{\gamma_1\subset \tilde V:\;\;\Delta_1\stackrel{\gamma_1}\leftrightarrow w_1}
  \Psi(\gamma_1;V){\bf 1}_{\{\gamma_1\cap C_1\ne \emptyset\}}
  \times
  \sum_{\gamma_2\subset \tilde V:\;\;\Delta_2\stackrel{\gamma_2}\leftrightarrow w_2}
  \Psi(\gamma_2;V).
\end{eqnarray}
Recalling \eqref{eq:denomina} one sees therefore that 
\begin{eqnarray}
\label{eq:A28}
    \frac{\cN_2''}{\Xi(\tau,V)}\le c \frac{Q}{(1-Q)^2},
\end{eqnarray}
where
\begin{eqnarray}
  Q:=\frac{\sum_{\{\gamma\subset \tilde V:\;\;
\Delta_1\stackrel{\gamma}\leftrightarrow w_1\}}
    \Psi(\gamma;V){\bf 1}_{\{\gamma\cap C_1\ne \emptyset\}}}  
  {\sum_{\{\gamma\subset \tilde V:\;\;\Delta_1\stackrel{\gamma}\leftrightarrow w_1\}}
    \Psi(\gamma;V)}
\end{eqnarray}
and we are left with the task of proving that $Q\le
\exp(-cL^{3\gep})$.  Note that $Q$ is nothing but the equilibrium
probability $\pi_V^{\hat\tau}(\gamma\cap C_1\ne\emptyset)$, where
$\gamma$ is the unique open contour for a system enclosed in $V$ and
with boundary conditions $\hat\tau$ given by $\hat\tau_x=+$ for
$x=(i,0)$ with $i<\lfloor L/2\rfloor-L^{3\gep}$ and $x=(0,i)$ with
$i\le2\lfloor (2L+1)^{1/2+\gep}\rfloor$, and $\hat\tau_x=-$ otherwise.
Morally, one would like to apply \cite[Th. 4.15]{cf:DKS} to say that
$Q\le \exp(-c L^{3\gep})$; such result however cannot be applied
directly because of the entropic repulsion effect that $\gamma$ feels
due to the South border of $V$, and we need to take a small detour.
Consider the $L$-shaped domain $W$ obtained as the union of the
rectangles $V$ and $V'$, where $V'=\{(i,j)\in\Z^2:-L^{1/2+\gep}\le
j\le 0,1\le i< \lfloor L/2\rfloor-L^{3\gep}-1\}$, with boundary
conditions $\hat\tau'$ given by $\hat \tau'=\hat \tau$ on $\partial
W\cap \partial V$ and $\hat \tau'=+$ on $\partial W\cap \partial V'$,
see Figure \ref{fig:W}.
\begin{figure}[htp]
\begin{center}
\includegraphics[width=0.8\textwidth]{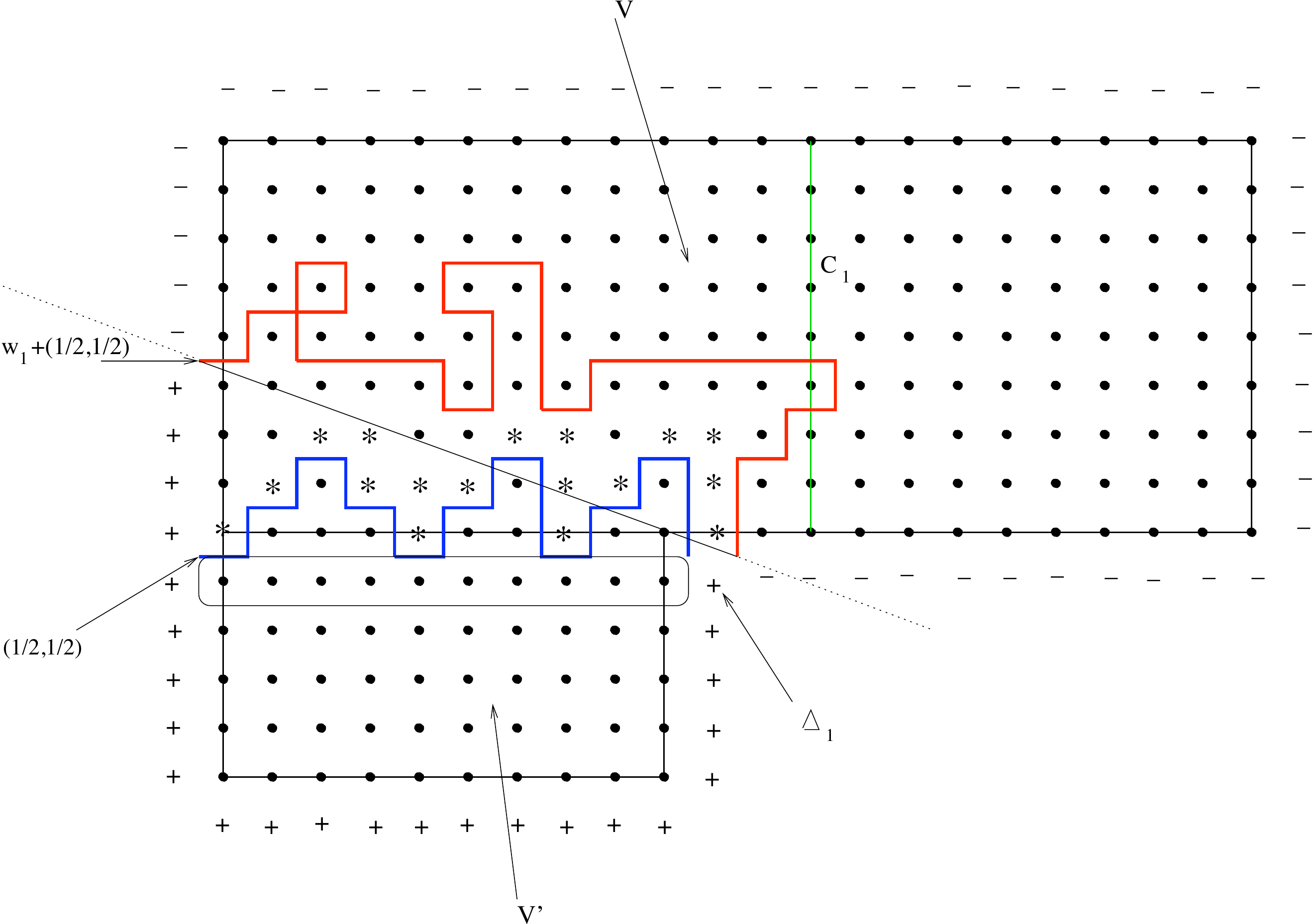}
\end{center}
\caption{The $L$-shaped domain $W$ (for graphical
  convenience, proportions are not respected in the drawing) with its boundary conditions
  $\hat\tau'$. For the construction of $\gamma'$, one should imagine that the spins in the framed region
are set to $-$. The sites marked by $*$ denote the $*$-connected set $\Delta^+(\gamma')$.
The drawn configuration of $\gamma$ is entirely above the straight line going through $w_1+(1/2,1/2)$ and $\Delta_1+(1/2,1/2)$,
\ie the spin configuration $\si$ belongs to the set $\Gamma''$ appearing in \eqref{eq:gamma''}.
}
\label{fig:W} 
\end{figure}

Below we will prove
\begin{Lemma}
\label{th:lemmaQ}
One has
\begin{eqnarray}
\label{eq:A30}
  Q=  \pi_V^{\hat\tau}(\gamma\cap C_1\ne\emptyset)\le 
  \pi_W^{\hat \tau'}(\gamma\cap C_1\ne\emptyset|\Gamma')
  \le\frac{\pi_W^{\hat \tau'}(\gamma\cap C_1 \ne \emptyset)}
{\pi_W^{\hat \tau'}(\Gamma')},
\end{eqnarray}
where  $\Gamma'=\{\si\in \O_W: \;\exists$  inside $V$ a
$*$-connected path of $+$ spins which connect the site 
$\Delta_1+(0,1)$ to one of the sites $(1,i)$ with $1\le i\le 2\lfloor (2L+1)^{1/2+\gep}\rfloor\}$, see Figure \ref{fig:W}.
\end{Lemma}
The numerator in the right-hand side of \eqref{eq:A30} is 
smaller than $\exp(-c L^{3\gep})$. Indeed, it suffices to
remark that (cf. the notation \eqref{eq:canonico}) it is smaller than
\begin{eqnarray}
\frac{\cE_{w_1,\Delta_1}\left[
{\bf 1}_{\{\gamma\cap C_1\ne\emptyset\}}
\exp\left(\Phi_{W}(\gamma)
\right)
\right]}{\cE_{w_1,\Delta_1}\left[\exp\left(\Phi_{W}(\gamma)\right)\right]}
\label{eq:cosciscvarz}
 \le \frac{\sqrt{\cP_{w_1,\Delta_1}(\gamma\cap C_1\ne\emptyset)
\cE_{w_1,\Delta_1}\left(
\exp\left(2 \Phi_{W}(\gamma) 
\right)
\right)
}}{\cE_{w_1,\Delta_1}\left[\exp\left(\Phi_{W}(\gamma)\right)\right]}.
\end{eqnarray}
where $\Phi_W(\gamma)$ was defined in \eqref{eq:8}.
Theorem 4.15 of \cite{cf:DKS} says directly  that $$\cP_{w_1,\Delta_1}(\gamma\cap C_1\ne\emptyset)\le\exp(-c L^{3\gep}),$$ while the
fast decay of $\Phi$, together with  \cite[Th. 4.16]{cf:DKS}, implies that 
\begin{eqnarray}
  \label{eq:2}
  &&\cE_{w_1,\Delta_1}\left[
\exp\left(2 \Phi_{W}(\gamma)
\right)
\right]\le \exp(c(\log L)^c)\\
\label{eq:155}
&& \cE_{w_1,\Delta_1}\left[\exp\left(\Phi_{W}(\gamma)\right)\right]\ge \exp(-c(\log L)^c).
\end{eqnarray}
Roughly speaking, typical paths (under
$\cP_{w_1,\Delta_1}$) have a small intersection with $W^c$ (again, the
precise estimates follow from \cite[Th. 4.15]{cf:DKS}). This is why we
enlarged $V$ to $W$: if $W$ were replaced by $V$, the intersection
 would not be small any more and the expectations in
\eqref{eq:2}-\eqref{eq:155} would not be under control.

The denominator in \eqref{eq:A30} is also not difficult to deal with:
one observes (see Figure \ref{fig:W}) that the event $\Gamma'$ is
implied by the event $\Gamma''=${\it $ \{\gamma$ does not go below the
  straight line which goes through $\Delta_1+(1/2,1/2)$ and $
  w_1+(1/2,1/2)\}$} (we will write symbolically $\gamma\ge(\Delta_1
w_1)$). Indeed, the subset of $\Delta\gamma$ where spins are $+$ is
$*$-connected and satisfies the requirements of $\Gamma'$.
Therefore, $\pi_V^{\hat \tau'}(\Gamma')\ge \exp(- cL^\gep)$.
Indeed,
\begin{eqnarray}
\label{eq:gamma''}
\pi_V^{\hat \tau'}(\Gamma')\ge 
\pi_W^{\hat \tau'}(\Gamma'')=\frac{\sum_{\gamma\sim \hat\tau'}
\Psi(\gamma;W){\bf 1}_{\{\gamma\ge (\Delta_1 w_1)\}}}{\sum_{\gamma\sim \hat\tau'}
\Psi(\gamma;W)}:
\end{eqnarray}
the numerator is lower bounded by $$\exp[-\beta\tau_\beta(\vec v_{w_1\Delta_1})
d(w_1,\Delta_1)-c(d(w_1,\Delta_1))^\gep]$$
via Lemma \ref{th:lemmastrazio} (take 
$\kappa=\gep/2$) and
the denominator is upper bounded by $$\exp[-\beta\tau_\beta(\vec v_{w_1\Delta_1})
d(w_1,\Delta_1)+c(\log d(w_1,\Delta_1))^c]$$
via \cite[Th. 4.16]{cf:DKS}, where $\vec v_{w_1\Delta_1}$ is the unit vector
pointing from $w_1$ to $\Delta_1$.

Summarizing, we have obtained $Q\le \exp(-c L^{3\gep})$ and, 
via \eqref{eq:A28} and \eqref{eq:A26}, we have proven \eqref{eq:tecnica2}.

{\sl Proof of Lemma \ref{th:lemmaQ}.}  Given a configuration
$\si\in\Omega_W,$ imagine to replace all its spins in $\partial V\cap V'$ by $-$, cf. Figure \ref{fig:W}; 
then, associated to the restriction $\si_V\in \Omega_V$,
there are exactly two open contours in $\tilde V$. 
The endpoints of these two contours are
$(1/2,1/2)$, $w_1+(1/2,1/2)$, $\Delta_1+(1/2,1,2)$ and
$\Delta_1+(-1/2,1/2)$. Under the assumption that $\si\in \Gamma'$, one
sees immediately that one of the two contours connects $w_1$ to
$\Delta_1$ (this is nothing else but the open contour which we have
called $\gamma$ so far, e.g. in \eqref{eq:A30}); we will call $\gamma'$ the second open
contour, see Figure \ref{fig:W}. Given a possible configuration for $\gamma'$, $V$ is divided
into two components, call them $V^\pm(\gamma')$, where $V^-(\gamma')$
is the one ``in contact with'' $V'$.  It is clear that the intersection
$\Delta^+(\gamma'):= \Delta\gamma'\cap V^+(\gamma')$ is a
$*$-connected set (\ie any two of its points can be linked by a
$*$-connected chain belonging to $\Delta^+(\gamma')$) and all spins
are $+$ there.  It is important to remark that if we take
$\si\in\Gamma'$ and flip any spin in
$V_{\gamma'}^{int}:=V^+(\gamma')\setminus \Delta^+(\gamma')$, the
configuration of $\gamma'$ does {\sl not} change.  Also, if (with abuse of
notation) we let $\pi_{\gamma'}$ denote the equilibrium measure in
$V_{\gamma'}^{int}$ with b.c. $+$ on the portion of the boundary which
coincides with $\Delta^+(\gamma')$ and $\hat \tau$ otherwise, one has
\begin{eqnarray}
\label{eq:burni}
  \pi_{\gamma'}(\gamma\cap C_1\ne\emptyset)\ge \pi_V^{\hat\tau}(\gamma
\cap C_1\ne\emptyset),
\end{eqnarray}
by FKG since the event $\gamma\cap C_1\ne\emptyset$ is increasing.
One has then, with $\cS$ the set of possible configurations of
$\gamma'$,
\begin{eqnarray}
  \frac{\pi_W^{\hat \tau'}(\gamma\cap C_1 \ne \emptyset)}
  {\pi_W^{\hat \tau'}(\Gamma')}&\ge& \frac{\pi_W^{\hat \tau'}
    (\gamma\cap C_1 \ne \emptyset;\Gamma')}
  {\pi_W^{\hat \tau'}(\Gamma')}\\\nonumber
 & =&\sum_{\xi\in\cS}{\pi_W^{\hat \tau'}
    (\gamma\cap C_1 \ne \emptyset|\Gamma';\gamma'=\xi)}
  \frac{\pi_W^{\hat\tau'}(\Gamma';\gamma'=\xi)}{\pi_W^{\hat \tau'}(\Gamma')}
  \\\nonumber&=& \sum_{\xi\in\cS} \pi_{\xi}(\gamma\cap C_1\ne\emptyset)
\frac{\pi_W^{\hat\tau'}(\Gamma';\gamma'=\xi)}{\pi_W^{\hat \tau'}(\Gamma')}\ge 
\pi_V^{\hat\tau}(\gamma
\cap C_1\ne\emptyset),
\end{eqnarray}
where we used \eqref{eq:burni} in the second inequality.
\qed

\subsubsection{A technical lemma}
Let $a:=(a_1,a_2)\in{\Z^2}^*$ and $b=(b_1,b_2)\in{\Z^2}^*$ with $b_1>a_1$.
Let $\vec v_{ab}$ be the unit vector pointing from $a$ to $b$ and
$\phi_{ab}$ be the angle which $\vec v_{ab}$ forms with $\vec e_1$.
Assume that $-\pi/4\le \phi_{ab}\le\pi/4$.  Let $A>0,\kappa>0$, let
$U_{a,b}=U_{a,b}(A,\kappa)\subset \bbR^2$ be the cigar-shaped 
region which is delimited by the two
curves
$$x\mapsto \xi^\pm_{a,b;A, \kappa}(x):=x\tan(\phi_{ab})\pm A
\left(\frac{(x-a_1)(b_1-x)}{b_1-a_1}
\right)^{1/2+\kappa},\;\; x\in[a_1,b_1],
$$
and $U^+_{a,b}$ be the upper half of $U_{a,b}$, obtained by slicing
$U_{a,b}$ along the segment $ab$.  Also, we will denote by
$\Sigma_{a,b}=\Sigma_{a,b}(A,\kappa)$ the set of all open contours
$\gamma$ having $a$ and $b$ as endpoints, and such that every bond in
$\gamma$ has non-empty intersection with $U_{a,b}$;  similarly we define
$\Sigma^+_{a,b}$. Then,
\begin{Lemma}
  \label{th:lemmastrazio}
  Let $\beta$ be large enough, and consider a domain $V\subset \Z^2$
  such that $\tilde V$ contains $U^+_{a,b}(A,\kappa)$ (cf. Definition 
\ref{def:VVbar}).  There exists $c$
  depending on $\beta,A,\kappa$ such that
\begin{eqnarray}
\label{eq:A34}
  \sum_{\gamma\in \Sigma^+_{a,b}}\Psi(\gamma;V)\ge 
  \exp\left[-\beta \tau_\beta(\vec v_{ab})d(a,b)-
c(d(a,b))^{2\kappa}
\right].
\end{eqnarray}
\end{Lemma}
This result can be obtained via a repeated use of Theorem 4.16 of
\cite{cf:DKS}.  The error term $\exp(-c\,(d(a,b))^{2\kappa})$ is very
rough (but sufficient for our purposes) and can presumably be  improved.
We do not give full details because they are a bit lengthy, although
standard, but we sketch the main steps. 

\begin{figure}[htp]
  \begin{center}
\includegraphics[width=0.9\textwidth]{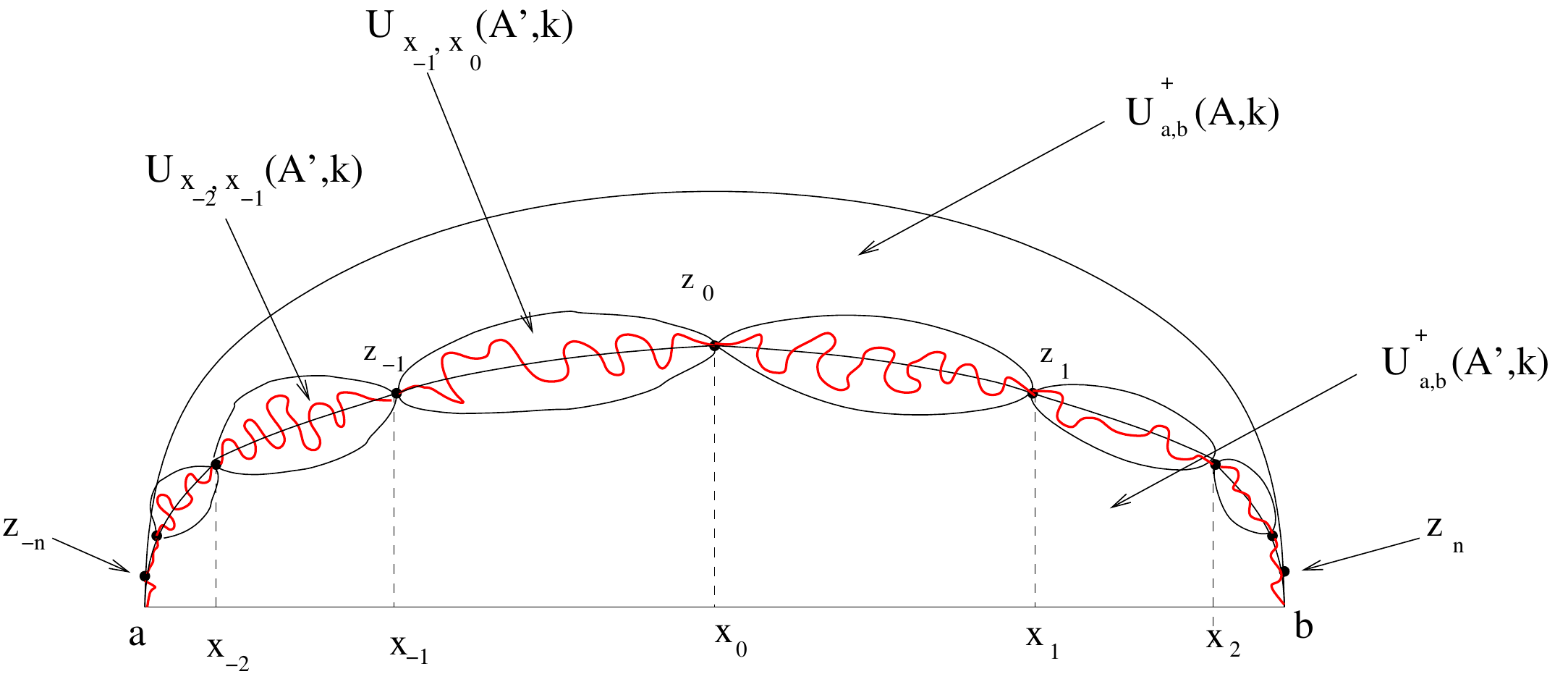}
\end{center}
\caption{A typical path $\gamma$ which contributes to the lower bound
  \eqref{eq:A34}.  For graphical convenience, we have assumed that $a$
  and $b$ have the same vertical coordinate, and not all the
  cigar-shaped sets $U_{z_i,z_{i+1}}(A',\kappa)$ have been drawn.  }
\label{fig:catenella} 
\end{figure}

First of all, let for simplicity of notations $L:=b_1-a_1$ and $A':=A/10$.
Then, one proceeds as follows (keep in mind Figure \ref{fig:catenella}):
\begin{itemize}
\item for every $-n\le i\le n$, with $n=\log_2(L)-2$, let
$z_i=(x_i,y_i)$  be a point in ${\Z^2}^*$ at minimal distance from
$(\tilde x_i,\xi^+_{a,b;A',\kappa}(\tilde x_i))$, where
\begin{eqnarray}
  \tilde x_i:=a_1+(b_1-a_1)
\left(\frac12+\frac{{\rm sign}(i)}4\sum_{j=0}^{|i|-1}2^{-j}\right);
\end{eqnarray}

\item remark via elementary geometrical considerations that for every
  $-n\le i<n$, the cigar-shaped set $U_{z_i,z_{i+1}}(A',\kappa)$ is
  entirely contained in $U^+_{a,b}(A,\kappa)$;

\item restrict the sum \eqref{eq:A34} to the paths $\gamma$ which,
  when oriented from $a$ to $b$, go through the points
  $z_{-n},z_{-n+1},\ldots, z_n$ (in this order), and such that the
  portion of the path between $z_{i}$ and $z_{i+1}$ belongs to
$\Sigma_{z_i,z_{i+1}}(A',\kappa)$;

\item remark that, via the decay properties of the potential $\Phi$, the 
interaction between two adjacent portions of $\gamma$ just defined can 
be bounded above by a constant;

\item apply Theorem 4.16 of \cite{cf:DKS} to write that for every
$-n\le i<n$ one has
\begin{eqnarray}
  \sum_{\gamma\in \Sigma_{z_i,z_{i+1}}(A',\kappa)}\Psi(\gamma;V)\ge
\exp\left[-\beta\tau_\beta(\vec v_{z_i,z_{i+1}})d(z_i,z_{i+1})-c
(\log d(z_i,z_{i+1}))^c\right],
\end{eqnarray}
for some constant $c$ depending on $A,\kappa,\beta$.
As for 
the two portions of $\gamma$ from $a$ to $z_{-n}$ and from $z_n$ to $b$,
they give a multiplicative contribution of order $1$  to \eqref{eq:A34}
(this is because $d(a,z_{-n})=O(1)$ and 
$d(b,z_{n})=O(1)$, as is immediately seen from the definition of $n$);

\item put together the estimates on the contributions coming from the $2n+3$
  portions of $\gamma$ obtained in the previous point: using the
  convexity and smoothness properties of the surface tension
  $\tau_\beta(\cdot)$, one obtains the claim of the lemma.
\end{itemize}

\section*{Acknowledgments}
We are extremely grateful to Senya Shlosman and to Yvan Velenik for
valuable help on low-temperature equilibrium estimates.  Part of this work was done during the authors' stay at the Institut Henri Poincar\'e - Centre Emile Borel
during the semester ``Interacting particle systems, statistical mechanics and probability theory''. The authors thank this institution for
 hospitality and support.

\end{document}